\documentclass[sn-mathphys,Numbered]{sn-jnl}

\usepackage{graphicx}%
\usepackage{multirow}%
\usepackage{multicol}   
\usepackage{amsmath,amssymb,amsfonts}%

\usepackage{mathrsfs}%
\usepackage[title]{appendix}%
\usepackage{xcolor}%
\usepackage{textcomp}%
\usepackage{manyfoot}%
\usepackage{booktabs}%
\usepackage{algorithm}%
\usepackage{algorithmicx}%
\usepackage{algpseudocode}%
\usepackage{listings}%
\usepackage{thm-restate}
\usepackage{dsfont}

\usepackage{thm-restate}
\usepackage{amsthm}

\usepackage{color}
\definecolor{clemson-orange}{RGB}{234,106,32}
\definecolor{chicago-maroon}{RGB}{128,0,0}
\definecolor{northwestern-purple}{RGB}{82,0,99}
\definecolor{cmu-red}{cmyk}{0,1.00,0.79,0.20}
\definecolor{cornell-red}{RGB}{179,27,27}
\definecolor{sauder-green}{RGB}{171,180,0}
\definecolor{harveymudd-gold}{RGB}{178,139,51}
\definecolor{hopkins-blue}{RGB}{0,45,114}
\definecolor{mit-red}{RGB}{163,31,52}
\definecolor{lawngreen}{RGB}{0,250,154}

\theoremstyle{thmstyleone}%
\newtheorem{theorem}{Theorem}
\newtheorem{lemma}{Lemma}
\newtheorem{proposition}{Proposition}
\newtheorem{corollary}{Corollary}
\newtheorem{claim}{Claim}

\theoremstyle{thmstyletwo}%
\newtheorem{example}{Example}%

\theoremstyle{thmstylethree}%
\newtheorem{definition}{Definition}%

\usepackage[nameinlink]{cleveref}
\crefname{assumption}{Assumption}{Assumptions}
\crefname{lemma}{Lemma}{Lemmas}
\crefname{theorem}{Theorem}{Theorems}
\crefname{corollary}{Corollary}{Corollaries}
\crefname{prop}{Proposition}{Propositions}
\crefname{claim}{Claim}{Claims}
\crefname{procedure}{Procedure}{Procedures}
\crefname{algorithm}{Algorithm}{Algorithms}
\crefname{figure}{Figure}{Figures}
\crefname{remark}{Remark}{Remarks}
\crefname{section}{Section}{Sections}
\crefname{procedure}{Procedure}{Procedures}
\crefname{proposition}{Proposition}{Propositions}
\crefname{example}{Example}{Examples}
\crefname{equation}{}{}
\crefname{enumi}{}{}
\crefname{table}{Table}{Tables}
\crefname{definition}{Definition}{Definitions}
\crefname{appendix}{Appendix}{Appendices}

\def\R{\mathbb{R}}

\def\L{\mathcal{L}}

\def\C{\mathcal{C}}

\def\S{\mathcal{S}}
\def\T{\mathcal{T}}

\DeclareMathOperator\supp{supp}

\DeclareMathOperator\cone{cone}
\DeclareMathOperator\relint{relint}
\DeclareMathOperator\conv{conv}

\usepackage{etoolbox}
\newcommand{\zerodisplayskips}{%
  \setlength{\abovedisplayskip}{5pt}%
  \setlength{\belowdisplayskip}{5pt}%
  \setlength{\abovedisplayshortskip}{5pt}%
  \setlength{\belowdisplayshortskip}{5pt}}
\appto{\normalsize}{\zerodisplayskips}
\appto{\small}{\zerodisplayskips}
\appto{\footnotesize}{\zerodisplayskips}

\begin{document}

\title[Designing optimization problems with diverse solutions]{
Optimizing for strategy diversity in the design of video games
}


\author[1]{\fnm{Oussama} \sur{Hanguir}}\email{oh2204@columbia.edu}

\author[2]{\fnm{Will} \sur{Ma}}\email{wm2428@gsb.columbia.edu}

\author[3]{\fnm{Christopher} \sur{Thomas Ryan}}\email{chris.ryan@sauder.ubc.ca}

\author[3]{\fnm{Jiangze} \sur{Han}}\email{jiangze.han@sauder.ubc.ca}

\affil[1]{\orgdiv{Industrial Engineering and Operations Research}, \orgname{Columbia University}, \orgaddress{\city{New York}, \postcode{NY 10027}, \country{USA}}}

\affil[2]{\orgdiv{Graduate School of Business}, \orgname{Columbia University}, \orgaddress{\city{New York}, \postcode{NY 10027}, \country{USA}}}

\affil[3]{\orgdiv{UBC Sauder School of Business}, \orgname{University of British Columbia}, \orgaddress{\city{Vancouver}, \postcode{V6T 1Z2}, \state{BC}, \country{Canada}}}


\abstract{We consider the problem of designing a linear program that has diverse solutions as the right-hand side varies. This problem arises in video game settings where designers aim to have players use different ``weapons'' or ``tactics'' as they progress. We model this design question as a choice over the constraint matrix $A$ and cost vector $c$ to maximize the number of possible \emph{supports} of unique optimal solutions (what we call ``loadouts'') of Linear Programs $\max\{c^\top x \mid Ax \le b, x \ge 0\}$ with nonnegative data considered over all resource vectors $b$. We provide an upper bound on the optimal number of loadouts and provide a family of constructions that have an asymptotically optimal number of loadouts. The upper bound is based on a connection between our problem and the study of triangulations of point sets arising from polyhedral combinatorics, and specifically the combinatorics of the cyclic polytope. Our asymptotically optimal construction also draws inspiration from the properties of the cyclic polytope.\footnote{An earlier version \cite{hanguir2023designing} of this paper appeared in the conference proceedings of IPCO 2023.}}




\maketitle

\section{Introduction}\label{s:introduction}

In this paper, we formulate the problem of \textit{designing} linear programs that allow for \textit{diversity} in their optimal solutions. For fixed consumption matrix $A\in\R^{m\times n}$, benefit vector $c\in\R^n$ and resource vector $b\in\R^m$, define $LP(A,c,b)$ as the following linear program
\begin{equation}\label{eqn:define-LP}
\max\{ c^\top x \mid Ax \leq b, x  \ge 0\},
\end{equation}
where $c$, $A$, and $b$ all have nonnegative data. If $x$ is the optimal solution of $LP(A,c,b)$, we define the support of $x$ as
$\supp(x) \triangleq \{i \in \{1,\ldots,n\} \mid x_i > 0 \}.$
If, for some right-hand side $b\in\mathbb{R}^m$, $x$ is the \emph{unique} optimal solution of $LP(A,c,b)$, then we call $\supp(x)$ an \emph{optimal loadout} (or simply, a loadout) of design $(A,c)$.

For fixed $n$ and $m$, the \textit{loadout maximization problem} is to choose the benefit vector $c$ and comsumption matrix $A$ that maximize the total number of loadouts of the design $(A,c)$. That is, the goal is to design the vector $c$ and the matrix $A$ so that the linear programs $LP(A,c,b)$ have as many possible supports of unique optimal solutions as possible, as resource vector $b$ varies in $\R^m$. We restrict attention to \emph{unique} optimal solutions. If we considered maximizing the number of supports of all (even non-unique) optimal solutions, this leads to degenerate cases that should be excluded. As a trivial example, in any design with $c = 0$, every loadout is optimal because they all lead to the same objective value of $0$.

Formally, $L \subseteq [n]$ is a loadout of design $(A,c)$ if there exists a nonnegative resource vector $b \in \R^m_{\geq 0}$ such that $LP(A,c,b)$ has a unique optimal solution $x^*$ with $\supp(x) = L$. We say that loadout $L$ is \emph{supported by} resource vector $b$. If $|L| = k$ then we say that $L$ is a $k$-loadout. Given a design $(A,c)$ and an integer $k \in [m]$, let $\L^k(A,c)$ denote the set of all $k$-loadouts of design $(A,c)$. The set of all loadouts of any size is $\L(A,c) \triangleq \cup_{k=1}^n \L^k(A,c)$.

Using this notation, we can restate the loadout optimization problem. Given dimensions $n$ and $m$ and integer $k \le n$, the \emph{$k$-loadout optimization problem} is 
\begin{equation}\label{eq:k-loadout-problem}
\max \{ |\mathcal L^k(A,c)| \mid A \in \R^{m \times n}, c \in \R^n, A \text{ and } c \text{ are nonnegative} \}.  \tag{$\text{L}_k$}
\end{equation}
%
We can assume without loss of generality that the linear programs $LP(A,c,b)$ are bounded  and thus possess an optimal solution because otherwise there is no optimal solution and, therefore, no loadout.

Given that a loadout corresponds to the support of a unique solution of a linear program,  any optimal solution with support size greater than $m$ cannot be unique. 
Therefore, the number of $k$-loadouts when $k > m$ is always equal to zero. This leads us to consider the optimization problems \cref{eq:k-loadout-problem} only for $k \in \{1, \ldots, \min(m,n)\}.$
For convenience, we will avoid the trivial case of $k=1$ where the optimal number of loadouts is $\min(m,n)$. The last case, which we immediately exclude, is when $\min(m,n) =n$, i.e. $m\ge n$. Here, a trivial design is optimal. By setting $A = I_n$ to be the identity matrix of size $n$, and $c = (1,\ldots,1)$, we ensure that for $k \in [1,n]$, every one of the $\binom{n}{k}$ subsets is a loadout (see  \cref{lemma:n_leq_m} in the appendix).

In summary, we proceed without loss of generality under the assumption that $n>m\ge k\ge2$.



\subsection{Video game motivation}

This setting is motivated by video games, in particular, the design of competitive games where players optimize their strategies to improve their in-game status. For such games, a desideratum for game designers is for optimizing players to play different strategies at different stages of the game.  


For many video games, engaged players aim to pick the best strategy available to conquer the challenges they face. Often, the key strategic decision is to select the best set of tools (often, weapons) to meet challenges. Players have limited in-game resources and face constraints (size of weapons, a limit on the number of weapons of some type, etc.) when selecting their strategy. These decisions, therefore, can be modeled as constrained optimization problems. 

To make this discussion concrete, consider the following fictional game. In RobotWar, we control a sci-fi robot to drive into battle against other robots. As we play the game, we accrue and manage experience points (XP). Before each battle, we pick the combination of weapons and equipment our robot takes into battle. There are different types of weapons including short-range (sabers, shotguns, etc.) and long-range (sniper cannon, explosive missiles, etc.). Weapons can be bought with experience points, and there are capacity constraints on the sizes of the weapons carried (including ammunition) that is proportional the amount of XP invested in them. Every weapon has  an initial damage (per usage) that it can deal to the opponent, but we can increase the damage of a weapon by investing additional XP. By winning more battles, we get more XP and increase the robot's capacity to hold more weapons.

Before each battle, a player picks the weapons that will be used, as well as how much XP is invested into each weapon. Assume that we are at a given stage of the game with a fixed amount of XP and fixed capacities. We want our robot to have the highest possible total damage value. We can compute the combination of weapons that maximizes the total damage by solving a linear program with decision variables representing how much XP is invested into each weapon. The set of weapons with a positive amount of invested XP is called the player's loadout. In the context of RobotWar, the key strategic choice is selecting a loadout. Note that a loadout refers to the combination of weapons and not the amount of XP invested in each weapon. If the same combination of weapons is adopted but with different allocations of XP, then in the both situations we have used the same loadout.

The previous example shows how a video game can be designed so that determination of an optimal loadouts can be done through optimization. In the next subsection, we present real-life video game settings where players, in fact, use optimization to form their strategies.

In light of the loadout decisions of players, game designers may ponder the following question: how to set the constraints of the game and the attributes of the tools so that the number of optimal loadouts, across all possible resource states, is maximized? In other words, the game designer may want to set the game up in such a way that, as the resources of the players evolve, the optimal loadouts change. A desire to maximize the number of optimal loadouts is motivated by the fact that video games can get boring when they are too repetitive with little or no variation \citep{schoenau2011player}. It is considered poor game design if a player can simply ``spam'' (i.e., repeatedly use) one strategy to progress easily through a game without needing to adjust their approach.

In our study, we assume that our game design question can be captured by a linear program. This is justified as follows. Consider a game that has $n$ available tools. The player has a decision variable $x_i$ for each tool $i$, which represents ``how much'' of the tool the player employs. We put ``how much'' in quotations because there are multiple interpretations of what this might mean. In the RobotWar example, $x_i$ denotes the amount of XP invested in a weapon $i$. The more XP is invested, the more the weapon can be used. This can be interpreted as a measure of ``ammunition'' or ``durability''.  Let $c \in \R_{\geq 0}^n$ denote the vector of benefits that accrue from using the various tools. That is, a single unit of a tool $i$ yields a per-unit benefit of $c_i$. In the RobotWar example, $c_i$ represents the damage dealt to the opponent by a weapon $i$ per a single unit of XP invested in a weapon $i$.

The player must obey a set of $m$ linear constraints when selecting tools. These constraints are captured by a matrix $A \in \R_{\geq 0}^{m \times n}$ and include considerations like a limit on the number of coins or experience points that the player has, a limit on the capacity (weight, energy, etc.) of the tools that can be carried, etc. A vector $b \in \R_{\geq 0}^m$ represents the available resources at the disposal of the player and forms the right-hand sides of the set of constraints. In the RobotWar example, there are natural constraints corresponding to total available XP and capacities for the various weapons. 

Therefore, we can interpret the loadout maximization problem as follows. We interpret the player's problem as solving a Linear Program of the form $\max\{ c^\top x \mid Ax \leq b, x \ge0\}$. Players at different stages of the game have different resource vectors $b$. The columns of $A$ correspond to the tools that the player can use in the game. We call a subset of these tools (represented by subsets of the columns of $A$) a \emph{loadout} (which literally means the equipment carried into battle by a soldier), if they correspond to the support of an optimal solution $x^*$ for some resource vector $b$ (in fact, we require $x^*$ to be the unique optimal solution of this linear program, for reasons that will become clear later). The support of a vector corresponds to a selection of the available tools, forming a strategy for how the player approaches the game given available resources. We assume that the game designer is able to choose $A$ and $c$. We refer to this choice as the \emph{design} of the game. We measure the \emph{diversity} of a design as the number of possible loadouts that arise as the resource vector $b$ changes. The game designer's problem is to find a design that maximizes the diversity. Then, a solution to this problem is able to meet the game designer's goal of finding a design, where the players use as many different loadouts as possible, while the game evolves and players' resources change. In the next subsection, we also provide a concrete game example to ground some of these concepts.

We highlight three key elements of this interpretation: 
\begin{itemize}
    \item We are interested in loadouts, corresponding to supports of optimal solutions, but not in the values of the solutions themselves. This is motivated by our desire to maximize the diversity; two solutions that have the same support and use the same set of tools but in different proportions are considered variations of the same strategy.
    \item We restrict our attention to the \emph{unique} optimal loadouts. The maximization of the supports of all (even non-unique) optimal solutions will lead us to degenerate cases that should be excluded. As a trivial example, in any design with $c = 0$, every loadout is optimal as they all give rise to the same objective value of $0$. This is not a desirable design in practice because there is no strategy involved for the player in designing their loadout. The consideration of only unique optimal solutions can also be motivated by the fract that it is helpful for a game designer to be able to predict the best strategy of a player given its resources. This can be especially useful in online role-playing video games and survival games, where the players fight computer-controlled enemies, and the knowledge of a player's optimal strategy helps to balance the difficulty of the environment.
    \item We only count the supports of unique \emph{optimal} solutions and ignore the supports of \emph{non-optimal} solutions. A reason for this has already been mentioned, the optimal behavior is easy to predict in contrast to ``non-optimal''. Furthermore, in many modern video games that have a ``freemium'' revenue model, most of the revenue comes from dedicated ``hardcore'' players, who are more likely to optimize in order to continue their in-game progress. By restricting attention to an optimal player behavior in the design problem, the game designer makes decisions for its highest tier of revenue-generating players.
\end{itemize}

\subsection{Some illustrative examples} 

We now turn to a consideration of some real games where the loadout maximization problem has relevance in practice. Although these examples may not perfectly fit the linear model that we study here in every aspect, they nonetheless share the same essential design and speak to the tradeoffs of interest.

Consider, for example, the MOBA (Multiplayer Online Battle Arena) genre that has become increasingly popular in past years. In the MOBA game \emph{SMITE}, players take control of deities from numerous pantheons across the world. Teams of deities work together to destroy objects in enemy territory, while defending their own territory against an enemy team trying to do the same. The main source of advantage is to purchase tools that enhance the attributes of the deities. These attributes include power, attack speed, lifesteal (percentage of damage dealt to enemy deities that is returned back to the player as health), and critical strike chance (the chance of an attack dealing double the damage that it would normally do). In this game, one can (and some players actually do \citep{smite}) use linear programming to model the tools to buy and decide which ones will give the best advantage, while at the same time keeping costs down. Game designers can, accordingly, anticipate the decision-making procedures of players and select the various attributes of the tools and their prices to promote the diversity of gameplay.

Another example of using linear programming to compute an optimal strategy is in the game \emph{Clash of Clans}. In this game, players fortify their base with buildings to obtain resources, create troops, and defend against attacks. Players put together raiding parties to attack other bases. Linear programming can be used to determine the best combination of characters in a raiding party, with constraints on the training cost of warriors (measured in the elixir, a resource that the player mines and/or plunders) as well as the space in the army camps to house units \citep{clashofclans}. 

While these examples show how the maximizing behavior among players can be effectively modeled as linear programs, there is also evidence that game designers are interested in maximizing the diversity using optimization tools. For example, the veteran game designer Paul Tozour presents the problem of diversity maximization in a series of articles on optimization and game design on Gamasutra, a video game development website \citep{supertank}. In one of his articles, Tozour describes the fictional game of ``SuperTank'' (similar to our fictional RobotWar) to show how optimization models can be used to design the attributes of available weapons that lead to varied styles of play. Tozour makes a strong case throughout his series of articles for using optimization tools, stating that game designers 
\begin{quote}
\dots might be able to use automated optimization tools to search through the possible answers to find the one that best meets their criteria, without having to play through the game thousands of times.
\end{quote}

\subsection{Our contributions}

We initiate the study of the loadout maximization problem. Our first contribution in the paper is to establish a link between the loadout maximization problem and the theory of polyhedral subdivisions and triangulations. In particular, for a fixed design $(A,c)$, the theory of triangulations offers a nice decomposition (or triangulation) of the cone generated by the columns of the constraint matrix $A$. This decomposition depends on the objective vector $c$. We show that, for a fixed design $(A,c)$, the loadouts can be seen as elements of this decomposition. This allows us to use a set of powerful tools from the theory of triangulations to prove structural results on the loadouts of a design. 

Our second contribution is to prove a non-trivial upper bound on the number of loadouts of any design. The upper bound involves an interesting connection to the faces of the so-called \emph{cyclic polytope}, a compelling object central to the theory of polyhedral combinatorics. We also show that this upper bound holds when the constraints of the linear program are equality constraints.

The third contribution of this paper is to present a construction of a design $(A,c)$ with a number of loadouts that asymptotically matches the above upper bound. Furthermore, for cases with few constraints, we present optimal constructions that \textit{exactly} match the upper bound. Our constructions provide practical insights that game designers can use to balance the tools available in the game, with the hope of increasing the diversity of strategies.


\subsection{Related work}\label{sec:relatedwork}

Our work is closely related to parametric linear programming, which is the study of how the optimal properties
depend on the data parameterization. The study of parametric linear programming dates back to the work of \cite{saaty1954parametric}, \cite{mills1956marginal},  \cite{williams1963marginal}, and \cite{walkup1969lifting} in the 1950s and 1960s. Parametric programming aims to understand the dependence of optimal solutions on one or more parameters; that is, on the entries of $A$, $b$, and $c$. In turn, our work is novel in the sense that our goal is to understand the structure of the supports of optimal solutions by fixing $A$ and $c$ and having $b$ vary in $\R^m_{\geq 0}$. To the best of our knowledge, this question has not previously been studied in the literature.

We should also note that we are not the first to look at questions of ``diversity'' in the context of optimization. For a recent survey of the various notions of diversity optimization see \cite{parreno2021measuring}. To the best of our knowledge, our paper is the first to study the loadout problem as we have defined it.
We are trying to design problems with diverse solutions, instead of searching for diverse optimal solutions of a given problem.

There have been several studies on the interface of optimization and video games, \cite{turner2011or,guo2019economic,sheng2020incentivized}. Guo \textit{et al.} \cite{guo2019selling} study the impact of selling virtual currency on player’s gameplay behavior, game provider’s strategies, and social welfare. Another significant research direction concentrates on studying ``loot boxes'' in video games. Chen \textit{et al.} \cite{chen2020loot} study the design and pricing of loot boxes, while Ryan \textit{et al.} \cite{ryan2020selling} study the pricing and deployment of enhancements that increase the player's chance of completing the game. Chen \textit{et al.} \cite{chen2017eomm} and Huang \textit{et al.} \cite{huang2019level} study the problem of in-game matchmaking to maximize a player's engagement in a video game.

\subsection{Organization of the paper}

In \Cref{s:formulation}, we cleanly state all of our results, sketch their proofs, and illustrate their intuition on small examples, without formally defining all
terminology and definitions related to linear programming and triangulations. These formal definitions can be later found in \Cref{sec:preliminaries}.
Our upper bound results are derived in \Cref{section:theorem1}, while our asymptotically optimal constructions are presented in \Cref{sec:thm2}.

\section{Statement of the main results
}\label{s:formulation}


In this section, we state our main results. To make these statements precise, we require some preliminary definitions. Let $[k]$ denote the set $\{1,\ldots,k\}$, for any positive integer $k$. Using this notation, we can define the support of $x \in \R^n_{\geq 0}$ as $\supp(x) = \{j \in [n] \mid x_j > 0 \}$. For a matrix $A\in\R^{m\times n}_{\ge0}$, the $(i,j)$-th entry is denoted $a_{ij}$ for $i \in [m]$ and $j \in [n]$, the $j$-th column is denoted $A_j$ for $j \in [n]$, and the $i$-th row is denoted $a_i$ (where $a_i$ is a column vector) for $i \in [m]$. 

\subsection{The Cyclic Polytope}\label{sec:cyclic-polytope-intro}

All of our bounds are intimately related to the number of faces of the \textit{cyclic polytope},
which is formally defined in \Cref{sec:preliminaries}.  
A remarkable aspect of the cyclic polytope is that for $n>m\ge2$, the cyclic polytope $\C(n,m)$ 
\textit{simultaneously} maximizes the number of $k$-dimensional faces for all $k=0,\ldots,m-1$ among $m$-dimensional polytopes over $n$ vertices, which is the property known as McMullen's Upper Bound Theorem \citep{mcmullen1970maximum}.
The number of $k$-dimensional faces on $\C(n,m)$ is given by the formula
\begin{equation*}
f_k(\mathcal{C}(n,m)) =  \sum\limits_{\ell = 0}^{\lfloor m/2 \rfloor}  \binom{\ell}{m-k -1} \binom{n - m + \ell -1}{\ell} + \sum\limits_{\ell = \lfloor m/2 \rfloor + 1}^m \binom{\ell}{m-k -1} \binom{n - \ell -1}{m - \ell}.
\end{equation*}
When $k=m-1$,  through the ``hockey stick'' identity on Pascal's triangle, this simplifies to
\begin{align*}
f_{m-1}(\mathcal{C}(n,m))
&=\binom{n - \lceil m/2\rceil}{\lfloor m/2\rfloor} + \binom{n - \lfloor m/2\rfloor -1}{\lceil m/2\rceil -1}.
\end{align*}

As an illustration of these formulas, suppose $m=3$.  The formulas evaluate to
\begin{align}
f_2(\mathcal{C}(n,3)) &=\binom{n-2}{1}+\binom{n-2}{1} && =&2n-4 \label{eqn:f2}
\\ f_1(\mathcal{C}(n,3)) &=1\binom{n-3}{1}+2\binom{n-3}{1}+3\binom{n-4}{0} &&=&3n-6 \label{eqn:f1}
\\ f_0(\mathcal{C}(n,3)) &=\binom{2}{2}\binom{n-3}{1}+\binom{3}{2}\binom{n-4}{0} &&=&n. \nonumber
\end{align}
To check that this is correct, note that $f_0(\C(n,3))$ should be $n$ by definition.  Meanwhile, we remark that the cyclic polytope is a \textit{simplicial} polytope, i.e.\ all of its $(m-1)$-dimensional faces are convex hull of exactly $m$ points. When $m=3$, this translates to all of its facets being triangles.  Therefore, $2f_2(\C(n,3))=3f_1(\C(n,3))$, since every edge is contained in exactly 2 triangles and every triangle contains exactly 3 edges.  In conjunction with Euler's immortal formula
$
f_2(\C(n,3)) + f_0 (\C(n,3)) = f_1(\C(n,3)) + 2,
$
one can \textit{uniquely} express $f_2(\C(n,3)),f_1(\C(n,3))$ as a function of $n$ for simplicial polytopes in 3 dimensions, which indeed can be checked to equal respective expressions~\cref{eqn:f2,eqn:f1} above.
In higher dimensions, simplicial polytopes can have different numbers of faces for each dimension, but they can never surpass the corresponding numbers for the cyclic polytope.

\subsection{Statements of Main Results}






\begin{theorem}\label{thm:upperbound}
Fix positive integers $n,m,k$ with $n>m\ge k\ge2$.
Then the number of $k$-loadouts for any design $(A,c)$ with $A\in\R^{m\times n}$ and $c\in\R^n$ satisfies
\begin{align} \label{eqn:introUB}
|\L^k(A,c)|
&\le f_{k-1}(\C(n+1,m))-\binom{m}{k-1}.
\end{align}
\end{theorem}

We note that the trivial upper bound on the number of $k$-loadouts in a design with $n$ tools is $\binom{n}{k}$.
When $m<n$, the RHS of \cref{eqn:introUB} will always be smaller than this trivial upper bound, which shows that having a limited number of resource types in the game does indeed prevent all subsets of tools from being viable.

\begin{theorem}\label{thm:lowerbound}
Fix positive integers $n,m,k$ with $n>m\ge k\ge2$.
Then we can provide a family of explicit designs $(A,c)$ with $A\in\R_{\ge 0}^{m\times n}$ and $c\in\R_{\ge 0}^n$ that satisfy
\begin{align*}
|\L^k(A,c)|&\ge\begin{cases}
f_{k-1}(\C(n,m)) & \text{if $k<m/2$}\\
f_{k-1}(\C(n,m))/2 & \text{if $k\ge m/2$}.
\end{cases}
\end{align*}
\end{theorem}


The constructions from \Cref{thm:lowerbound} are always within a 1/2-factor of being optimal asymptotically as $n\to\infty$ because it is known (see  \cref{lemma:asymptotic_fk} in Appendix~\ref{appx:asymptotic_fk} for a formal proof) that
\begin{align*}
\lim_{n\to\infty}\frac{f_{k-1}(\C(n,m))}{f_{k-1}(\C(n+1,m))}=1.
\end{align*}

\begin{restatable}{theorem}{lowerboundSmallM}\label{thm:lowerboundSmallM}
For $n>m=3$, we can provide a family of explicit designs $(A,c)$ with $A\in\R_{\ge 0}^{m\times n}$ and $c\in\R_{\ge 0}^n$ that satisfy $|\L^3(A,c)| \ge 2n-5$ and $|\L^2(A,c)| \ge3n-6$.
\end{restatable}

\begin{restatable}{theorem}{lowerboundSmallMtwo}\label{thm:lowerboundSmallMtwo}
For $n>m=2$, we can provide a family of explicit designs $(A,c)$ with $A\in\R^{m\times n}$ and $c\in\R^n$ that satisfy
$|\L^2(A,c)| \ge n-1.$
\end{restatable}

The constructions from \Cref{thm:lowerboundSmallM} and \Cref{thm:lowerboundSmallMtwo} are \textit{exactly tight}; it can be checked that they match the upper bound expression from \Cref{thm:upperbound} when evaluated at $m=3$ and $m=2$. The proofs of both theorems are deferred to Appendix \ref{appx:exact_construction}.

\vskip 10pt

\textbf{Example of the construction} from \Cref{thm:lowerbound} and intuition. \cref{table:construction} shows an example of the asymptotically optimal construction for $m = 4$ and $n = 6$. 
Our construction provides a pattern that game designers can follow to diversify loadouts on a set of tools $1,\ldots,n$, by having two types of resource constraints. In the first set of constraints (corresponding to rows 1 and 3 of the matrix) tools with smaller indices consume more resources. In the second set of constraints (corresponding to rows 2 and 4) tools with large indices consume more resources. 

Furthermore, among the constraints of the same type, the resource requirements of tools either monotonically increase or monotonically decrease along the rows. The implication of this is as follows. Consider the tool corresponding to the first column in \cref{table:construction}. This tool is cheapest with respect to the second and fourth resources and the most expensive with respect to the first and third. This monotone structure heightens the sensitivity of the structure of optimal solutions to changes in the resource vector. Practically speaking, this means that tools that are very powerful in some dimensions must also have significant weaknesses to ensure a variety of play. A concrete example of this is the ``rocket launcher'' in first-person shooters, which is typically the most powerful weapon but suffers from having the most expensive ammunition. 

This ``tension'' between the two types of constraints, and the fact that each tool has strengths and weaknesses, means that a given tool is unlikely to be in many loadouts. This captures the rough intuition that a game with an overpowered tool (meaning one that is more useful than the others but also  not significantly ``cumbersome'' to limit its use) leads to uniform strategies among players: they always use the overpowered tool.

\begin{table}
\centering
\begin{tabular}{|c|c|c|c|c|c|c|}
\hline
$c$                  & 1    & 1    & 1    & 1    & 1    & 1    \\ \hline
\multirow{4}{*}{$A$} & $M-1$    & $M-2$    & $M-3$    & $M-4$    & $M-5$    & $M-6$   \\ 
& $1^2$ & $2^2$ & $3^2$ & $4^2$ & $5^2$ & $6^2$ \\ 
                   & $M-1^3$    & $M-2^3$    & $M-3^3$   & $M-4^3$   & $M-5^3$  & $M-6^3$  \\ 
                   & $1^4$ & $2^4$ & $3^4$ & $4^4$ & $5^4$ & $6^4$ \\ \hline
\end{tabular}
\caption{Example of our construction with $m = 4$, $n = 6$, and $M = 6^3 + 1$.}
\label{table:construction}
\end{table}

\subsection{Roadmap for proving \Cref{thm:upperbound,thm:lowerbound}}\label{ss:roadmap}

This subsection provides a high-level overview of our approach for establishing our upper and lower bounds. All the undefined terminology used here will be defined in more detail in later sections.

We prove our upper bound \cref{thm:upperbound} using a sequence of transformations.
\begin{enumerate}
\item We first introduce the intermediate concept of an \textit{equality loadout} problem that replaces the inequality constraint $Ax \leq b$ with an equality $Ax = b$. Accordingly, we define \emph{$k$-equality loadout} as $k$-loadout in this revised problem. We show that for a fixed design $(A,c)$ and for every dimension $k$, the number of $k$-loadouts is less than the number of $k$-equality loadouts (\Cref{prop:inclusion}).

\item This allows us to focus on proving an upper bound on the number of equality loadouts. Here, we can exploit the dual structure of the equality LP and prove that equality loadouts belong to a cell complex $\Delta_c(A)$ that is characterized by $A$ and $c$.
Importantly, we show that loadouts correspond to \textit{simplicial} cells in this cell complex (\Cref{lemma:loadout-simplicial}).

\item In turn, this allows us to, without loss of generality, assume that $\Delta_c(A)$ is a \textit{triangulation} (as opposed to an arbitrary subdivision), of a cone in the positive orthant of $\R^m$ (\Cref{lemma:refinement}).

\item We show that triangulations of cones in the positive orthant of $\R^m$ correspond to triangulations of points in the lower dimensional space $\R^{m-1}$ (\Cref{lem:pointedcone}).

\item Finally, we show that the simplices in this triangulation can be embedded into faces of a simplicial polytope in $\R^m$. Therefore, any upper bound on the number of faces of polytopes in $\R^m$ implies an upper bound on the number of loadouts. This allows us to invoke the ``maximality'' of the cyclic polytope with respect to its number of faces mentioned in \cref{sec:cyclic-polytope-intro}. 
Therefore, the number faces of the cyclic polytope of dimension $m$ bounds the number of faces in a polytope of dimension $m$, and implies a bound on the number of equality loadouts.
We also carefully count the number of extraneous faces added through our transformations, by invoking a bound on the minimal number of faces a polytope can have, which allows us to derive tight bounds for small values of $m$ (\Cref{lemma:embedding}).
\end{enumerate}

We remark that in the above proof, we need first to map to triangulations in $\R^{m-1}$, and later, return to polytopes in $\R^m$ in order to invoke McMullen's Upper Bound.
However, this required the introduction of a point at the south pole, which means that it is difficult for our upper bound to be tight for small values of $n$.
Nonetheless, the introduction of this additional point is insignificant as $n\to\infty$. This is why we can prove asymptotic optimality.

To prove our complementing lower bound \Cref{thm:lowerbound}, we first explicitly provide our design $(A,c)$ in \Cref{sec:construction},
which is also inspired by the cyclic polytope. Compared to the cyclic polytope, every even row of the matrix $A$ has been ``flipped'' (as previously observed in \Cref{table:construction}), for reasons that will become apparent in our proof, which we now outline.
\begin{enumerate}
\item First, we focus on the dual program of $LP(A,c,b)$ and present a sufficient condition (\Cref{def:inequalitycell}) for loadouts in terms of dual variables (\Cref{lem:inequalitycell}).
\item We show that by taking hyperplanes corresponding to the facets of the cyclic polytope in dimension $m$, one can attempt to construct dual variables that satisfy the sufficient condition (\Cref{lem:hyperplaneequation}).
Our aforementioned ``flipping'' of rows in $A$ is crucial to this construction of the dual variables.
We show that, as long as the facet of the cyclic polytope is of the ``even'' parity, the constructed dual variables will indeed be sufficient (\Cref{lemma:remaining_conditions}), and hence such a facet and all of the faces contained within it correspond to loadouts.




\item Therefore, to count the number of $k$-loadouts, we need to count the number of $(k-1)$-dimensional faces of a cyclic polytope in dimension $m$ that are contained within at least one even facet.
To the best of our knowledge, this problem has not been directly solved in the literature.
Nonetheless, using Gale's evenness criterion we can map this to a purely combinatorial problem on binary strings (\cref{lemma:face_array} and \cref{lemma:facets_intersection}).  Through some combinatorial bijections, we show that we lose at most a factor of 2 by restricting to even facets, and lose nothing if $k$ is small.  These arguments form the cases in \cref{thm:lowerbound}.
We note that it was important for us to be interested in even facets instead of odd facets, as counting odd facets can lose an unboundedly large factor.
\end{enumerate}

To summarize, both our upper and lower bounds employ the cyclic polytope, but through different transformations: projecting down to $\R^{m-1}$ and then lifting back up for the upper bound, and ``flipping'' rows for the lower bound.

\begin{figure}
    \centering
    \includegraphics[scale=0.5]{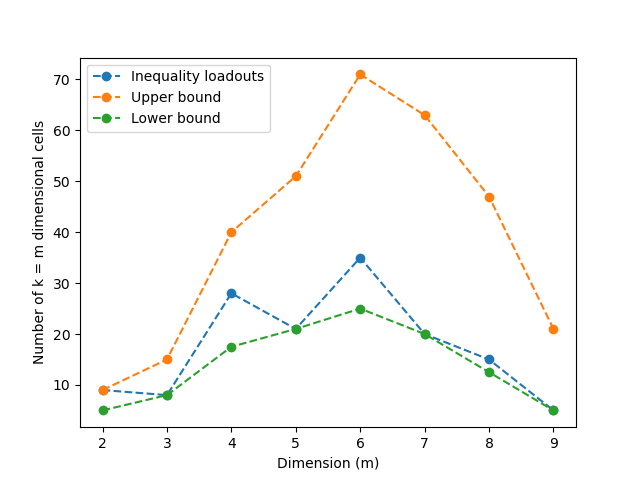}
    \caption{An illustration of upper bound in \cref{thm:upperbound} and lower bound in \cref{thm:lowerbound}.}
    \label{fig:construction}
\end{figure}

\vskip 10pt

Finally, we end this subsection with the results of a set of numerical experiments that tested the strength of our lower and upper bounds. \cref{fig:construction} considers a setting where $n=10$. We construct the matrix $A$ according to \cref{sec:construction} for different dimensions $m$. The cost vector $c$ is $(1,\ldots, 1)$. For a given design $(A,c)$, we enumerate all triangulations by the software Sagemath and select our design $(A,c)$ from the list. Then, we test each cell in the corresponding triangulation to see if it is an (inequality) loadout or not. As we can see, the lower bound constructed in \cref{thm:lowerbound} is close to counting the number of loadouts in our construction (it was always exact when $m$ was odd) and the upper bound is often appeared to be conservative. 

\section{Preliminaries}\label{sec:preliminaries}

We present terminology we use in the proofs of both \cref{thm:upperbound,thm:lowerbound}. Additional terminology needed in the proof of only one of these results could be found in the relevant sections.

A \emph{$d$-simplex} is a $d$-dimensional polytope that is the convex hull of $d +1$ affinely independent points. For instance, a 0-simplex is a point, a 1-simplex is a line segment
and 2-simplex is a triangle.
For a matrix $A = (A_1,\ldots, A_n)$ of rank $m$, let $\cone(A) = \cone(\{A_1, \ldots, A_n\})$ represent the closed convex polyhedral cone $\{Ax \mid x \in \R^{n}_+\}$. We use the notation $\cone(C)$ to denote the cone generated by the columns indexed by $C \subseteq [n]$.  If $C \subseteq [n]$ is a subset of indices, the \textit{relative interior} of $C$ \cite[Definition 2.5.1]{de2010triangulations} is defined to be the relative interior of the conic hull spanned by $\{A_j\}_{j\in C}$, i.e., the relatively open (i.e., open in its affine hull) cone
    \[ \relint_A(C) \triangleq \big\{ \sum\limits_{j\in C} \lambda_j A_j  \mid  \lambda_j > 0  \mbox{ for all } j \in C\}. \]
A subset $F$ of polytope $P$ is a \textit{face}  if there exists $\alpha \in \R^n$ and $\beta \in \R$ such that $\alpha^\top x + \beta \leq 0$ for all $x \in P$ and $F = \{x \in P \mid \alpha^\top x + \beta\ = 0\}$. If $\dim(F) = k$ then $F$ is called a \textit{$k$-dimensional face} or \textit{$k$-face}. The faces of dimensions 0, 1, and $\dim(P) - 1$ are called \emph{vertices}, \emph{edges}, and \emph{facets}, respectively. Furthermore, we say that $F$ is face of $C$, where $F, C \subseteq [n]$, when $\cone(F)$ is a face of $\cone(C)$. We define a polyhedral subdivision of $\cone(A)$ as follows.
\begin{definition}\label{def:cone_subdivision} \cite[Definition 2.5.7]{de2010triangulations}
Let $A = (A_1,\ldots,A_n)$ be a matrix of rank $m$. A collection $\mathscr{S}$
of subsets of $[n]$ is a \textbf{polyhedral subdivision} of $\cone(A)$ if it satisfies the following conditions:
\begin{itemize}
\setlength{\itemindent}{1.5em}
    \item[(CP)] If $C \in \mathscr{S}$ and $F$ is a face of $C$, then $F \in \mathscr{S}$. (Closure Property)
    \item[(UP)] $\cone(\{1,\ldots,n\}) \subset \bigcup\limits_{C \in \mathscr{S}} \cone(C)$. (Union Property)
    \item[(IP)] If $C, C' \in \mathscr{S}$ with $C \neq C'$, then $\relint_A(C) \cap \relint_A(C') = \emptyset$. (Intersection Property)
\end{itemize}
\end{definition}


If $\{j_1, \ldots,j_k\} $ belongs to a subdivision of $\cone(A)$, then the set of indices $\{j_1, \ldots,j_k\}$ is called a \textit{cell} of the subdivision, and if the cone is of $k$-dimensional, it is called a $k$-cell. We note that a polyhedral cone subdivision is completely specified by listing its maximal cells.


Next, we define a special subdivision of $\cone(A)$ as a function of the cost vector $c$. The cells of this subdivision map to the loadouts of the design $(A,c)$. For $A\in\R^{m\times n}_{\ge0}$ and $c\in\R^n_{\ge0}$, we define the polyhedral subdivision $\Delta_c(A)$ of $\cone(A)$ as a family of subsets of $\{1,\ldots, n\}$ such that $C \in \Delta_c(A)$ if and only if there exists a column vector $y\in \mathbb{R}^m$ such that $y^\top A_j = c_j$ if $j \in C$ and $y^\top A_j > c_j$ if $j \in \left\{1,\dots, n\right\} \setminus C$. In such a case, we say that $C$ is a cell of $\Delta_c(A)$ and that $\Delta_c(A)$ is a \textit{cell complex}. A cell $C \in \Delta_c(A)$ is simplicial if the column vectors $(A_j)_{j \in C}$ are linearly independent. If all the cells of $\Delta_c(A)$ are simplicial, then we say that $\Delta_c(A)$ is a \emph{triangulation}. The maximum size of a simplicial cell is $m$. The next results shows that $\Delta_c(A)$ is indeed a polyhedral subdivision of $\cone(A)$. The proof is deferred to Appendix~\ref{appendixA}.

\begin{proposition}\label{lemma:3properties}
$\Delta_c(A)$ is a polyherdal subdivision of $\cone(A)$.
\end{proposition}

Intuitively, we can think of the subdivision $\Delta_c(A)$ as follows: take the cost vector $c$,  and use it to lift the columns of $A$ to $\R^{m+1}$ then look at the projection of the upper faces (those faces you would see if you ``look from above''). This is illustrated in \cref{ex:lifting-example}.



\begin{example}\label{ex:lifting-example}
Consider the following matrix and cost vectors
\begin{equation}\label{eqn:example-data}
A = \begin{pmatrix}
1/4 & 1/2 & 3/4\\
1 & 1 & 1
\end{pmatrix}, \quad c_1 = (2,2.125 + \epsilon,2.25) \quad \mbox{ and } \quad c_2 = (2,2.125-\epsilon,2.25),
\end{equation}
where $\epsilon > 0$ is a small constant. The corresponding subdivisions of $\cone(A)$ are
\[\Delta_{c_1}(A) = \big\{ \{1,2\}, \{2,3\}, \{1\}, \{2\}, \{3\}, \emptyset\big\} \ \ \mbox{ and } \ \ \Delta_{c_2}(A) = \big\{ \{1,3\}, \{1\}, \{3\}, \emptyset\big\}.\]

\noindent For example, to see that $\{1,2\}$ is a cell of $\Delta_{c_1}(A)$, we consider $y = (0.5 + 4\epsilon, 1.875-\epsilon)$. One can verify that $y^\top A_1 = c_1$ and $y^\top A_2 = c_2$, while $y^\top A_3 > c_3$. 
See  \cref{fig:lifting} for a visualization of $\Delta_{c_1}(A)$ and $\Delta_{c_2}(A)$. $\triangleleft$
\end{example}

\begin{figure}
    \centering
    \includegraphics[scale=1]{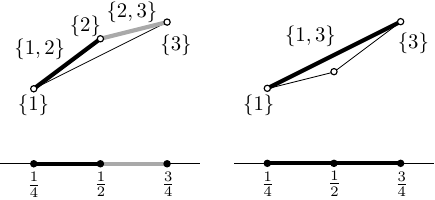}
    \caption{An illustration of the triangulations $\Delta_{c_1}(A)$ (left side of the figure) and $\Delta_{c_2}(A)$ (right side of the figure). In the figure, the third dimension (corresponding to the row of $1$'s in the matrix $A$ in \cref{eqn:example-data}) is suppressed since all objects are at the same height of $1$.}
    \label{fig:lifting}
\end{figure}


In our definition of simplicial cell, we mentioned that if all the cells in the subdivision $\Delta_c(A)$ are simplicial, then $\Delta_c(A)$ is called a triangulation. More generally, a triangulation of cones is a cone subdivision where all the cells are simplicial (the columns of every cell are linearly independent). We will also define the notion of triangulations of point configurations, which we define below. 

\begin{definition}
Let $B = \{x_1,\ldots,x_n\}$ be a point configuration (i.e., a finite set $B$ of points) in $\R^n$. A \textbf{triangulation} of $B$ is a collection $\mathcal{T}$ of simplices 
whose vertices are points in $B$, and whose dimension is the same dimension as the affine hull of $B$, with the following properties:
    \begin{itemize}
    \setlength{\itemindent}{1.5em}
    \item[(CP)] If $C \in \T$ and $F \subseteq C$, then $F \in \T$. (Closure Property)
    \item[(UP)] $\conv(B) \subset \bigcup\limits_{C \in \T} \conv(C)$. (Union Property)
    \item[(IP)] If $C, C' \in \T$ with $C \neq C'$, then $\relint_B(C) \cap \relint_B(C') = \emptyset$. 
    (Intersection Property)
\end{itemize}
\end{definition}



As described in the roadmap for \cref{thm:upperbound} in \cref{ss:roadmap}, to prove the upper bound on the number of loadouts, we will show that the cells of a triangulation and, therefore, the loadouts can be seen as faces of a higher dimensional polytope and that any upper bound on the number of faces of that polytope implies an upper bound on the number of loadouts. A crucial part of our analysis invokes the ``maximality'' of the cyclic polytope with respect to its number of faces, as described already in \cref{sec:cyclic-polytope-intro}. 
Now, we present a formal definition of the cyclic polytope as well as the $f$-vector of a polytope, that contains all the information about the number of faces.


 
\begin{definition}[Cyclic Polytope]\label{def:cyclicpolytope}
The \textbf{cyclic polytope} $\mathcal{C}(n,d)$ may be defined as the convex hull of $n$ distinct vertices on the moment curve $t \mapsto (t, t^2,\ldots, t^d)$. The precise choice of which $n$ points on this curve are selected is irrelevant for the combinatorial structure of this polytope. See an illustration of a cyclic polytope in \cref{fig:cyclic_polytope}.
\end{definition}

\begin{definition}[$f$-vector]\label{def:fvector}
The \textbf{$f$-vector} of a $d$-dimensional polytope $P$ is given by $(f_0(P), \ldots , f_{d-1}(P))$, where $f_i(P)$
enumerates the number of $i$-dimensional faces in the $d$-dimensional polytope for all $i=0,\ldots,d-1$. For instance, a 3-dimensional cube has eight vertices, twelve edges, and six facets, so its $f$-vector is $(f_0(P),f_1(P),f_2(P))=(8,12,6)$.
\end{definition}

As stated earlier, \cite{mcmullen1970maximum} shows that that the cyclic polytope $\mathcal{C}(n,d)$ maximizes the number of faces over every dimension for convex polytopes in dimension $d$. In other words, for any $d$-dimensional polytope $P$ on $n$ vertices, we have $f_i(P) \leq f_i(\mathcal{C}(n,d+1))$ for  $1 \leq i \leq d.$ This result is known as McMullen's Upper Bound Theorem.




\begin{figure}[ht]
    \centering
    \includegraphics[scale=1.3]{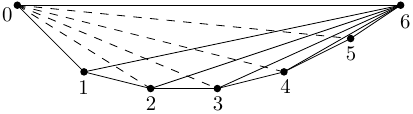}
    \caption{Representation of the cyclic polytope $\mathcal{C}(7,3)$.}
    \label{fig:cyclic_polytope}
\end{figure}

\section{Upper Bound (Proof of \Cref{thm:upperbound})}\label{section:theorem1}

Throughout this section we fix positive integers $n > m\ge 2$ and $A\in\R^{m\times n}_{\ge0},c\in\R^n_{\ge0}$.
We start by formally introducing the \textit{equality} loadout problem. We consider the parametric family of linear programming problems with equality constraints
\begin{equation*}
LP_=(A,c,b): \qquad \max\{ c^\top x \mid Ax = b, x  \ge 0\},
\end{equation*}
By analogy to the definition of loadouts in \cref{s:formulation}, an \textit{equality loadout} is defined as a subset of indices $L \subseteq \left\{1,\dots, n\right\}$ such that there exists a resource vector $b$ for which $LP_{=}(A,c,b)$ has a unique optimal solution $x^*$ such that $\supp(x^*) = L$. If $|L| = k$ then we say that $L$ is a $k$-equality loadout. Given $A$ and $c$ and an integer $k \in [m]$, let $\L^k_{=}(A,c)$ denote the family of all equality loadouts $L$ of dimension $k$. Finally, $\L_{=}(A,c)$ denotes the family of equality loadouts of all dimensions given $A$ and $c$. Namely, $\L_{=}(A,c) \triangleq \cup_{k=1}^m \L^k_{=}(A,c)$. The following proposition bounds the number of loadouts by the number of equality loadouts, for fixed $A$ and $c$.

\begin{lemma}\label{prop:inclusion} For every $A\in\R^{m\times n}_{\ge0}, \ c\in\R^n_{\ge0}$ and $k \in [m]$ we have $\L^k(A,c)\subseteq\L^k_{=}(A,c)$.
\end{lemma}
\noindent \proof 
Consider $L \in \L^k(A,c)$. There exists $b \in \R^m_{\geq 0}$ and $x \in \R^n_{\geq 0}$ with $\supp(x) = L$ such that $x$ is the unique optimal solution of $LP(A,c,b)$. We can see that $x$ is also the unique optimal solution of $LP_=(A,c,b')$ where $b' = Ax$. Any other optimal solution to $LP_=(A,c,b')$ would also be optimal for $LP(A,c,b)$. 
\endproof

\begin{example}[The case when the inequality is strict]
    Consider the following design $(A,c)$:
    \begin{equation*}
        A = \begin{pmatrix}
\frac{1}{4} & \frac{1}{2} & \frac{3}{4} & 1\\
1 & 1 & 1 & 1\\
\end{pmatrix} \quad c=(2,2,2.25,1)^\top.
    \end{equation*}
  The equality loadouts of size $k = 2$ are $\{\{1,3\}, \{3,4\}\}$ while the only one such inequality loadout is $\{1,3\}$.
\end{example}

In the rest of this section, assume without loss of generality that A is a full-row rank matrix. To see that this assumption is not restrictive, let $A$ be an arbitrary $m\times n$ non-negative matrix and let $A^f$ be a submartix of $A$ containing a maximal set of linearly independent rows of $A$. One can see that any equality loadout of $A$ is an equality loadout of $A^f$. Therefore, $\L^k_{=}(A,c)\subseteq\L^k_{=}(A^f,c)$, and since our objective in this section is to provide an upper bound on the number of loadouts, we may assume that $A$ is of full row rank.

We present, for all $k\in[m]$, an upper bound for the number $|\L^k_{=}(A,c)|$ of equality loadouts of size $k$ with respect to the design $(A,c)$. To do so, we divide the cone corresponding to the columns of $A$ into a collection of cells of $\Delta_c(A)$. 
We show that loadouts correspond to simplicial cells in $\Delta_c(A)$ and that we can restrict ourselves, without loss of generality, to designs $(A,c)$ where all the cells of $\Delta_c(A)$ are simplicial.
Finally, we present an upper bound on the number of cells of any dimension $k$ in a triangulation, which yields an upper bound on $|\L^k_{=}(A,c)|$.

Some of the results of this section are known in the literature (an excellent reference is the textbook \cite{de2010triangulations}), but we present them using our notation and adapted to the loadout terminology. We provide proofs for clarity and of our a desire to be as self-contained as possible. The proofs are also suggestive of some aspects of our later constructions in \cref{sec:thm2}.


\subsection{From equality loadouts to triangulations}





The following result links the optimal solutions of $LP_=(A,c,b)$ with the cells of subdivision $\Delta_c(A)$.

\begin{proposition}\label{prop:sturfmels-thomas}(\cite{sturmfels1997variation}, Lemma 1.4)
The optimal solutions $x$ to $LP_=(A,c,b)$ are the solutions to the problem 
\begin{equation}\label{eq:support-subset-of-cell}
\text{Find } x \in \mathbb{R}^n \text{ s.t. } Ax =b, x \ge 0, \text{ and } \text{supp}(x) \text{ is a subset of a cell of } \Delta_c(A). 
\end{equation}
\end{proposition}


\begin{lemma}\label{lemma:loadout-simplicial}
A subset $L \subseteq [n]$ is a loadout of $(A,c)$ if and only if it is a simplicial cell in the subdivision $\Delta_c(A)$.
\end{lemma}
\noindent \proof
Suppose $L$ is simplicial cell of $\Delta_c(A)$, and let $y$ be a vector corresponding to that cell (in the sense of the discussion preceding \cref{lemma:3properties}). Set the right-hand side $b = \sum_{j \in L} \alpha_i A_j$ for some $\alpha_j >0$, $\forall j \in L$. We show that $L$ is an equality loadout by showing that $\Bar{x} = (\Bar{x}_L, \Bar{x}_{\Bar{L}}) = (\alpha,0)$ (where $\bar L = [n] \setminus L$) is the unique optimal solution of $L_=(A,c,b)$. We first show that $\bar{x}$ is optimal. Note that $\Bar{x}$ and $y$ are respectively primal and dual feasible, and they satisfy the complementary slackness conditions. In fact, since $A\bar{x} =b$ by definition, we have $y_i(a_i^\top x-b_i)=0$ for $i \in [m]$. Furthermore, by definition of $y$, we have $y^\top A_j = c_j$ for $j \in L$, and since $\supp(x) = L$, we have $x_j = 0$ for all $j \not\in L$, which implies $(y^\top A_j - c_j)x_j = 0.$ This shows that $\Bar{x}$ and $y$ satisfy the complementary slackness conditions. Therefore, $\Bar{x}$ (resp. $y$) is primal (resp. dual) optimal. We now show that $\bar{x}$ is unique. Suppose now that there is another solution $x'$ to $L_=(A,c,b)$. Then $x'$ and $y$ verify the complementary slackness conditions. This implies that $x'_j=0$ for $j \not\in L$, and $\bar{x}$ and $x'$ have support in $L$. But since $L$ is simplicial, the columns $(A_j)_{j\in L}$ are linearly independent, and the only solution to $Ax = b$ with support in $L$ is $\bar{x}$. Therefore, $\bar{x}=x'$.
    
Assume now that $L$ is a loadout for a right-hand side $b$. By  \cref{prop:sturfmels-thomas} there exists a cell $C \in \Delta_c$ such that $L \subset C$. Suppose that $L$ is not a cell of $\Delta_c$. By   \cref{lemma:3properties},  $\Delta_c(A)$ is subdivision of $\cone(A)$. Therefore, by the property (CP) in \cref{def:cone_subdivision}, $L$ is not a face of any cell. Furthermore, since $L$ is a loadout, there exists a solution $x$ such that $\supp(x) = L$ and $Ax = b$. This implies that $b \in \relint(\cone(L)) \subset \cone(C)$. All faces of $C$ are cells, and by Corollary 11.11(a) in \cite{soltan2019lectures}, the following equality holds 
    \begin{equation*}
        \cone(C) = \bigcup \{ \relint(\cone(F))  \mid F \text{ is a face of } \cone(C) \}.
    \end{equation*}
 Therefore, $b$ lies in the interior of some face $F$, and by \cref{prop:sturfmels-thomas}, $F$ contains the support of an optimal solution for $L_=(A,c,b)$. Because we assumed that $L$ is not a cell, then $F \neq L$. This contradicts the uniqueness of the support $L$ that is required for $L$ to be a loadout. Therefore, $L$ is a cell of $\Delta_c(A)$. Assume now that $L$ is not simplicial and $L = \{j_1,\ldots,j_k\}$, this means that there exists $\gamma_2, \ldots,  \gamma_{k}$ such that, wlog, 
    \[
    \sum\limits_{i=2}^{k} \gamma_i A_{j_i} = A_{j_1} \ \ \mbox{ and } \ \
    \sum\limits_{i=2}^{k} \gamma_i c_{j_1} = c_{j_1}.\]
    Note that the $\gamma_i$ need not to be all positive. Consider $\alpha > 0$ such that $b = \alpha_1 A_{j_1} + \ldots +\alpha_{k} A_{j_{k}}$ and such that $x= (\alpha,0)$ is an optimal solution for $L_=(A,c,b)$. Let $\alpha_{min} = \min\limits_{i \in \{1,\ldots,k\}} \alpha_i$, $\gamma_{min} = \max\limits_{i \in \{1,\ldots,k\}} |\gamma_i|$, and $\epsilon = \frac{\alpha_{min}}{\gamma_{min}}$. It is clear that $\alpha_i \geq \epsilon \gamma_i$ and $\epsilon > 0$. We can rewrite the right-hand side $b$ as follows:
    \begin{align*}
    b = (\alpha_1+\epsilon) A_{j_1} +\sum\limits_{i=2}^{k} (\alpha_i - \epsilon \gamma_j) A_{j_i}.
\end{align*}
We can therefore define a new solution $x'$ such that $x'_{j_i} = \alpha_i - \epsilon \gamma_i$ for $i \in \{2,\ldots, k\}$, $x'_{j_{1}} = \alpha_1 + \epsilon$, and $x'_j = 0$ otherwise. We claim that $x$ and $x'$ have the same cost. In fact,
\begin{align*}
    c^\top x & = \sum\limits_{i=1}^{k} \alpha_i c_{j_i};\\
    c^\top x' & = (\alpha_1+\epsilon) c_{j_{1}} + \sum\limits_{i=2}^{k} (\alpha_j-\epsilon\gamma_j) c_{j_i}  = \epsilon(c_{j_1} - \sum\limits_{i=2}^{k} \gamma_j c_{j_i}) + \sum\limits_{i=1}^{k} \alpha_j c_{j_i}  = c^\top x.
\end{align*}
This contradicts the uniqueness of the loadout $L$. Thus $L$ is a simplicial cell of $\Delta_c(A).$
\endproof

The lemma above implies that we can focus on the simplicial cells of the subdivision $\Delta_c(A)$. We next show that we can consider without loss of generality choices of $c$ where all the cells of $\Delta_c(A)$ are simplicial. The idea is that if $\Delta_c(A)$ has some non-simplicial cells, then we can ``perturb'' the cost vector $c$ to some $c'$ and transform at least one non-simplicial cell into one or more simplicial cells. This perturbation conserves all the simplicial cells of $\Delta_c(A)$ and thus the number of equality loadouts for the design $(A,c')$ cannot be less than the number of equality loadouts for the design $(A,c)$. Without loss of optimality, we can ignore cost vectors $c$ that give rise to non-simplicial cells. First, we define the refinement that formalizes the ``perturbation'' of $c$.


\begin{definition}
Given two cell complexes $\mathcal{C}_1$ and $\mathcal{C}_2$, we say that $\mathcal{C}_1$ refines $\mathcal{C}_2$ if every cell of $\mathcal{C}_1$ is contained in a cell of $\mathcal{C}_2$.
\end{definition}


In \cite[Lemma 2.3.15]{de2010triangulations}, it is shown that, if $c'$ is a perturbation of $c$ that is $\epsilon$-close to $c$, i.e., $|c'_i-c_i|\leq\epsilon$, with $\epsilon > 0$ sufficiently small, then the new subdivision $\Delta_{c'}(A)$ refines $\Delta_c(A)$. Since $\Delta_{c'}(A)$ refines $\Delta_{c}(A)$, then $\Delta_{c'}(A)$ will have more cells. However, it is not clear if $\Delta_{c'}(A)$ will have more simplicial cells than $\Delta_{c}(A)$. We show in the following lemma that this is the case. We show that such a refinement preserves all the simplicial cells of $\Delta_c(A)$, and can only augment the number of simplicial cells.


\begin{lemma}\label{lemma:refinement}
A refinement of $\Delta_c$ can only add to the number of simplicial cells in $\Delta_c$.
\end{lemma}
\proof 
We fix the matrix $A$ and let $\Delta_c$ denote $\Delta_c(A)$. Assume $\Delta_c$ is not a triangulation. There exists $\epsilon > 0$, such that for every cost vector $c'$ that verifies $|c_i-c'_i| \leq \epsilon$, $\Delta_{c'}$ is a refinement of $\Delta_c$, i.e., for every cell $C' \in \Delta_{c'}$, there exists a cell $C \in \Delta_c$ such that $C' \subset C$. Now, we will argue that all simplicial cells of $\Delta_c$ are simplicial cells of every refinement $\Delta_{c'}$. Let $F$ be a  simplicial cell of $\Delta_c$. Let $x$ be a point in the relative interior of $F$. There exists a cell $C' \in \Delta_{c'}$ such that $x \in \relint(C')$, and furthermore $\dim C' = \dim F$.  By definition of a refinement there exists $C \in \Delta_c$ such that $C' \subseteq C$ and $x \in \relint(C') \subset \relint(C)$. Therefore, $C$ and $F$ are both cells of the subdivision $\Delta_c$ and $\relint(C) \cap \relint(F) \neq \emptyset$. This implies that $C = F$ by the intersection property. We have established that, for every simplicial cell $F$ in $\Delta_c$, there exists a maximal cell $C'$ in $\Delta_c'$ such that $C' \subseteq F$. Since $F$ is simplicial, $C'$ is a face of $F$, and the closure property says $C'$ is a cell of $\Delta_c$. Furthermore, since $\dim C' = \dim F$ and $C' \subseteq F$, then $C' = F$ and $F$ is a simplicial cell of the refinement $\Delta_{c'}$.
\endproof

In \cite[Corollary 2.3.18]{de2010triangulations}, it is shown that $\Delta_c(A)$ can be refined to a triangulation within a finite number of refinements (suffices for $c'$ to be generic). Therefore, the lemma above implies that in order to maximize the number of loadouts for any dimension $k \leq m$, we can restrict attention to designs $(A,c)$ such that $\Delta_c(A)$ is a triangulation without loss of generality.
 
We observe that since the matrix $A \in \R_{\geq 0}^{m\times n}$
has all nonnegative entries, $\cone(A)$ is contained entirely in the positive orthant and therefore cannot contain a line. Cones that do not contain lines are called \emph{pointed}. The following lemma shows that triangulations of pointed cones in dimension $m$ are equivalent to triangulations of a non-restricted set of points (columns) in dimension $m-1$. This implies that equality loadouts (which are shown to correspond to cells in a cone triangulation) can be seen as cells of a triangulation of a point configuration. The proof of the lemma is deferred to  Appendix~\ref{appx:lemma_pointed_cone}.

\begin{lemma}[\cite{beckcomputing}, Theorem 3.2]\label{lem:pointedcone}
Every triangulation $\T$ of a pointed cone of dimension $m$ can be considered as a triangulation $\T'$ of a point configuration of dimension $m-1$ such that for $1 \leq k \leq m$, the $k$-simplices of $\T$ can be considered as $(k-1)$-simplices of $\T'$.
\end{lemma}


\cref{lem:pointedcone} implies that equality loadouts of dimension $k$ correspond to $(k-1)$-simplices in a triangulation of a point configuration in dimension $m-1$.

\subsection{From cells of a triangulation to faces of a polytope}

Recall that $n>m\ge k\ge2$. We have just shown that the number of equality $k$-loadouts is upper-bounded by the maximum possible number of $(k-1)$-simplices in a triangulation of $n$ points in $\mathbb{R}^{m-1}$. We now show that any $n$-point triangulation in $\mathbb{R}^{m-1}$ can be embedded onto the boundary of an $(n+1)$-vertex polytope in $\R^m$, in a way such that $(k-1)$-simplices in the triangulation correspond to $(k-1)$-faces on the polytope.
We then apply the cyclic polytope upper bound on the number of $(k-1)$-faces of any $(n+1)$-vertex polytope in $\mathbb{R}^m$.
To get a tighter bound, we carefully subtract the ``extraneous'' faces added from the embedding that did not correspond to $(k-1)$-simplices in the original triangulation. We construct a lower bound on the number of such extraneous faces using the lower bound theorem of \cite[Theorem 1.1]{kalai1987rigidity}. 

Let $\T$ denote the original $n$-point triangulation in $\R^{m-1}$.  We will use $\conv \T$ to refer to the polytope obtained by taking the convex hull of all the faces in $\T$. Let $g_{k-1}(\T)$ denote the number of $(k-1)$-simplices in the triangulation $\T$. We embed $\conv \T$ into a polytope $P$ in $\R^m$ as follows.
Let $z^1,\ldots,z^n\in\R^{m-1}$ denote the vertices in the triangulation $\T$.
We now define the following lifted points in $\R^m$.
For all $i=1,\ldots,n$, let $\underline{z}^i$ denote the point $(z^i_1,\ldots,z^i_{m-1},0)$.
For all $i=1,\ldots,n$, let $\Bar{z}^i$ denote the point $(z^i_1,\ldots,z^i_{m-1},\epsilon)$, for some fixed $\epsilon>0$.
 Let $\epsilon >0$, and replace each point $\underline{z}^i$ that is in the interior of $
\conv(\{\underline{z}^1,\ldots,\underline{z}^n\})$ by the ``lifted'' point $\Bar{z}^i = (z^i_1, \ldots, z_{m-1}^i,\epsilon)$. The points on the boundary of $\conv(\{\underline{z}^1,\ldots,\underline{z}^n\})$ are not lifted. Let $S$ be the set of the $n$ points in $\R^m$ after lifting. Let $\S_m$ be the unit sphere in $\R^m$ with center at the origin, and $S'$ be the projection of $S$ onto $\S_m$, where every point is projected along the line connecting the point to the center of the sphere. The set $S'$ has the property that all the points that are on the ``equator" hyperplane $z_{m} = 0$ are exactly the projections of the points
of $S$ on the boundary of $\conv(S)$ (the points that were not lifted). The other points of $S'$ are in the ``northern hemisphere'' (the half space $x_{m}>0$). The final step is to adjoin the boundary points to the  ``south pole'',$(0 ,\ldots ,0, -1) \in \R^{m}$. Let $P$ be the resulting polytope, i.e., $P = \conv(S' \cup (0 ,\ldots ,0, -1))$.

The next lemma shows that for $2 \leq k \leq m$, the $(k-1)$-dimensional faces of $P$ are either $(k-1)$-simplices of $\T$, or $(k-2)$-faces of $\T$ that were adjoined to the south pole.


\begin{lemma}\label{lemma:embedding}
For $2 \leq k \leq m$, we have
$f_{k-1}(P) =g_{k-1}(\T)+f_{k-2}(\T).$
\end{lemma}
\proof 
Fix $k \in \{2,\ldots,m\}$. The projection of every $(k-1)$-simplex of $\T$ (after lifting the non-boundary points) is a simplicial face of $P$. Let $F$ be a $(k-2)$-dimensional face of $\conv \T$. The points of $F$ lie on the boundary of $\T$, and by adjoining them to the south pole, we create a $(k-1)$-face of the new polytope $P$. 
\endproof

The previous lemma implies that $g_{k-1}(\T) = f_{k-1}(P)- f_{k-2}(\T)$. Since $P$ has $n+1$ points, we know from the upper bound theorem that $f_{k-1}(P) \leq f_{k-1}(\C(n+1,m))$. Therefore, $g_{k-1}(\T) \leq f_{k-1}(\C(n+1,m))- f_{k-2}(\T)$, and all we need is a lower bound  on $f_{k-2}(\T)$. The following lemma uses the lower bound theorem \cite[Theorem 1.1]{kalai1987rigidity} to establish a lower bound  on $f_{k-2}(\T)$. The lower bound theorem presents a lower bound on the number of faces in every dimension among all polytopes of dimension $d$ over $p$ points, for $d \geq 2$ and $p \geq 2$.

\begin{lemma}\label{lemma:upper_bound_triangulation}
For $2 \leq k \leq m$, we have
$g_{k-1}(\T) \leq f_{k-1}(\C(n+1,m))-\binom{m}{k-1}.$
\end{lemma}
\proof  Let $p$ denote the number of vertices (boundary points) of the polytope $\T$.  
By the lower bound theorem of \cite{kalai1987rigidity}, we obtain
\[
f_{k-2}(\T)\ge\left\{
    \begin{array}{ll}
        \binom{m-1}{k-2}p - \binom{m}{k-1}(k-2) & \mbox{if }   k=2,\ldots,m-1, \\ \\
        (m-2)p - m(m-3) & \mbox{if }  k = m.
    \end{array}
\right .
\]
The right-hand side is  increasing in $p$.  But the minimum possible value of $p$ is $m$ (since $\conv \T$ is a full-dimensional  polytope in $\R^{m-1}$).  Hence
\[
f_{k-2}(\T)\ge\left\{
    \begin{array}{ll}
        \binom{m-1}{k-2}m - \binom{m}{k-1}(k-2) & \mbox{if }   k=2,\ldots,m-1, \\ \\
        m & \mbox{if }  k = m.
    \end{array}
\right .
\]
We observe that $\binom{m-1}{k-2}m - \binom{m}{k-1}(k-2)$ evaluates to $m$ if $k=m$.  Therefore, we can combine the two cases and derive using \cref{lemma:embedding} that
\begin{align*}
g_{k-1}(\T)
&\le f_{k-1}(P)-\left(\binom{m-1}{k-2}m - \binom{m}{k-1}(k-2)\right)
\\ &\le f_{k-1}(\C(n+1,m))-\left(\frac{m!}{(k-2)!(m-k+1)!} - \frac{m!}{(k-1)!(m-k+1)!}(k-2)\right)
\\ &= f_{k-1}(\C(n+1,m))-\left(\frac{m!}{(k-1)!(m-k+1)!}(k-1) - \frac{m!}{(k-1)!(m-k+1)!}(k-2)\right)
\\ &= f_{k-1}(\C(n+1,m))-\binom{m}{k-1},
\end{align*}
where we used the fact $f_{k-1}(P) \leq f_{k-1}(\C(n+1,m))$ from the upper bound theorem. 
\endproof

\subsection{Proof of Theorem 1}
We are now ready to present the proof of \cref{thm:upperbound}.

\begin{proof}[Proof of \cref{thm:upperbound}]
Consider $k \in \{2,\ldots,m\}$, \cref{lemma:loadout-simplicial} states that equality loadouts of size $k$ are $k$-cells in the cone subdivision $\Delta_c(A)$. By \cref{lemma:refinement}, $\Delta_c(A)$ can be considered as a triangulation of cones and by \cref{lem:pointedcone}, the number of $k$-cells $\Delta_c(A)$ is less than the maximum number of $(k-1)$-cells in a triangulation of $n$ points in dimension $m-1$. 

Finally, \cref{lemma:upper_bound_triangulation} shows that the number $(k-1)$-cells in a triangulation of $n$ points in dimension $m-1$ is less than $f_{k-1}(\C(n+1,m)) - \binom{m}{k-1}$. Therefore, $|\L^k_=(A,c)| \leq f_{k-1}(\C(n+1,m)) - \binom{m}{k-1}$. This inequality, combined with \cref{prop:inclusion}, yields
\begin{equation*}
|\L^k(A,c)| \leq |\L^k_=(A,c)| \leq f_{k-1}(\mathcal{C}(n+1,m)) - \binom{m}{k-1}. 
\end{equation*}
\end{proof}

\section{General Lower Bound (Proof of \Cref{thm:lowerbound})}\label{sec:thm2}


Throughout this section, we fix positive integers $n>m\ge2$, and explicitly present general designs $(A,c)$ that have the number of $k$-loadouts promised in \Cref{thm:lowerbound} for all $k\leq m$. For $m=2$ and $m=3$, better designs that are exactly optimal can be found in Appendix \ref{appx:exact_construction}. All of the designs constructed in this paper will satisfy the property that $A$ has linearly independent rows, hence we assume in the rest of this section that $A$ is a full row rank matrix.

\subsection{Construction based on moment curve} \label{sec:construction}

Let $t_1,\ldots,t_n$ be arbitrary real numbers satisfying $0< t_1 < t_2 < \ldots < t_n$. 
Let $M$ be an arbitrary constant satisfying $M\ge t^m$.
We define the design $(A, c)$ so that 
$c = (1, \ldots, 1) \in \mathbb{R}^n \mbox{ and } A = [v'_m(t_1), \ldots v'_m(t_n)]$,
where
\[t \mapsto v'_m(t) = \left(t, M-t^2, t^3, M-t^4, \dots, t^m\right)^\top \in \mathbb{R}^m,\] if $m$ is odd; and
\[t \mapsto v'_m(t) = \left(M-t, t^2, M-t^3, t^4, \dots, M-t^{m-1}, t^m\right)^\top \in \mathbb{R}^m,\] if $m$ is even.

For any such values $t_1,\ldots,t_n$ and $M$, we will get a design that satisfies our \Cref{thm:lowerbound}.
We set all the entries of the cost vector $c$ to 1 to simplify computations. It is not a requirement and the construction would still hold by setting $c_j$ to be any positive number and scaling the column $A_j$ by a factor of $c_j$.

\textbf{Motivation behind the construction.}
Let $P$ be the convex hull of $\{v'_m(t_1), \ldots v'_m(t_n)\}$.
Let
\[ t \mapsto v_m(t) = (t, t^2, t^3, \ldots, t^m)^\top \in \mathbb{R}^m\]
denote the $m$-dimensional moment curve that defines the cyclic polytope.

The choice of the curve $v'_m$ is motivated by role the cyclic polytope plays in our corresponding upper bound \cref{thm:upperbound}. In fact, \cref{thm:upperbound} shows that the number of $k$-dimensional  loadouts is less than the number of $(k-1)$-dimensional faces of the cyclic polytope $\mathcal{C}(n+1,m)$ (for $2 \leq k \leq m$). An ideal lower bound proof would connect the number of  loadouts to the number of faces of the cyclic polytope. However, simply setting the columns of the constraint matrix $A$ to be points on the moment curve of the cyclic polytope does not guarantee the existence of loadouts. We therefore, introduce the curve $v'_m$ that describes a ``rotated'' cyclic polytope and show that it is rotated to ensure that the supporting normals of ``half'' of the facets are nonnegative. We use these rotated facets to construct a number of loadouts that asymptotically matches the upper bound. The rotation is performed by multiplying the even coordinates of the moments curve by -1, and we use a sufficiently big constant $M$ to ensure the positivity of the new constraint matrix.



\subsection{Dual certificate for loadouts} \label{sec:dualCert}





Using LP duality, we derive a sufficient condition for subsets of $[n]$ to be loadouts.

\begin{definition}\label{def:inequalitycell}
A set $C\subseteq[n]$ is an \textbf{inequality cell} of the design $(A,c)$ if there exists a variable $y\in\mathbb{R}^m$ such that
\begin{alignat}{4}\label{eq:inequalitycell}
y_i & >  0  , \quad \forall \ i \in [m];\\
\nonumber y^\top A_j & =  c_j  ,  \quad \forall \ j \in C;\\
\nonumber y^\top A_j & >  c_j  ,  \quad \forall \ j \not\in C. 
\end{alignat}
\end{definition}

Here, $y$ can be interpreted as a dual variable.
However, in contrast to the definition of a cell that features in \Cref{prop:sturfmels-thomas}, here we require $y>0$.
This is because non-negativity is needed for $y$ to be feasible in the dual when the LP has an inequality constraint $Ax\le b$ instead of an equality constraint as considered in \cref{prop:sturfmels-thomas}.

\begin{lemma}\label{lem:inequalitycell}
Suppose $C\subseteq[n]$ is an inequality cell with $|C|=m$.  Then every non-empty subset of $C$ is a loadout.
\end{lemma}

 To establish \Cref{lem:inequalitycell}, we show that for every subset $L \subseteq C$, $y$ will verify the complementary slackness constraints with a primal variable $x$ that has support equal to $L$. This establishes the optimality of $x$, and to show its uniqueness, we use the assumption that $A$ has a full row rank equal to $m$.

\begin{proof}[Proof of \Cref{lem:inequalitycell}]
\noindent We start by recalling the complementary slackness conditions. If $x$ and $y$ are feasible solutions to the primal and dual problem, respectively, then complementary slackness states that $x$ and $y$ are optimal solutions to their respective problems if and only if
\begin{alignat*}{4}\tag{CS}
  y_i(a_i^\top x-b_i) & = & \ 0, \ & \quad \forall \ i \in [m],\\
   (c_j - y^\top A_j) x_j & = & \ 0, \ & \quad \forall \ j \in [n].
\end{alignat*}
Now, let $L$ be a non-empty subset of $C$.  We must show that $L$ is a loadout. Take an arbitrary $x^L\ge0$ with support equal to $L$, and define $b$ to be equal $Ax^L$.
Since $C$ is an inequality cell by \cref{def:inequalitycell}, there exists a dual variable $y^C$ satisfying the conditions in \cref{eq:inequalitycell}. Consider $LP(A,c,b)$. Clearly $x^L$ and $y^C$ are primal and dual feasible. They also satisfy the complementary slackness conditions. Therefore, $x^L$ and $y^C$ are primal and dual optimal. We now argue that $x^L$ is the unique optimal solution of $LP(A,c,b)$. If $x$ is not unique, there exists another optimal solution $x'$. By complementary slackness, $x'$ and $y^C$ must satisfy 
   $(c_j - (y^C)^\top A_j) x'_j  =  \ 0$ for all $j \in [n].$

By definition of $y^C$,  $(y^C)^\top A_j  >  c_j,  \ \forall \ j \not\in C$. Therefore, $\supp(x') \subseteq C$. The other complementary slackness condition $y_i(a_i^\top x-b_i)  = 0$  for $i \in [m],$ implies that 
\begin{align}\label{eq:full_rank_system}
[A_{j_1}|\cdots|A_{j_m}] \ x'=b,
\end{align}
where $C=\{j_1,\ldots,j_m\}$. But since $A$ is assumed to be of full rank, the columns $A_{j_1}, \cdots, A_{j_m}$ are linearly independent and the system \cref{eq:full_rank_system} has a unique solution. Since we have
\begin{align*}
[A_{j_1}|\cdots|A_{j_m}] \  x^L =b,
\end{align*}
by definition of $x^L$ and $b$, $x^L$ is the unique solution to \cref{eq:full_rank_system} and therefore the unique optimal solution of $LP(A,c,b)$. This shows that $L$ is a loadout and concludes the proof.  
\end{proof}


\subsection{Deriving dual certificates for our construction} \label{sec:reducingToCombinatorial}

In order to prove \cref{thm:lowerbound}, we
consider our design from \Cref{sec:construction}, and try to show that there are many inequality cells of cardinality $m$.
To do so, we take an arbitrary $C\subseteq[n]$ with $|C|=m$ and consider the hyperplane that goes through the $m$ points $\{v'_m(t_j) \mid j\in C\}$.
We show in \Cref{lem:hyperplaneequation} that the coefficients of the equation for this hyperplane have the same sign. We then use these coefficients to construct a candidate dual vector $y$. The last step (\cref{lemma:remaining_conditions}) is to show that when the hyperplane satisfies a \textit{gap parity }combinatorial condition, this dual vector will indeed satisfy \Cref{def:inequalitycell}, certifying that $C$ is an inequality cell. The proofs are presented in Appendix~\ref{appendix:proof_lemmas}.

\begin{lemma}\label{lem:hyperplaneequation}
Let $C = \{j_1, \ldots, j_m\} \subseteq [n]$ be a subset of $m$ indices with $j_1<\cdots<j_m$. The equation 
\begin{equation}\label{eq:hyperplane}
 \det\begin{pmatrix}
1 & \ldots & 1 & 1\\
v'_m(t_{j_1}) & \ldots & v'_m(t_{j_m}) & y \end{pmatrix} = 0
\end{equation}
defines a hyperplane in variable $y\in\mathbb{R}^m$ that passes through the points $v'_m(t_{j_1}), \ldots, v'_m(t_{j_m}).$ Furthermore if equation \cref{eq:hyperplane} is written in the form
\[ \alpha_1 z_1 + \ldots \alpha_m z_m - \beta = 0, \]
then we have $\alpha_1 \neq 0, \ldots \alpha_m \neq  0, \beta \neq 0$, and 
\[ \mbox{sign}(\alpha_1) = \ldots = \mbox{sign}(\alpha_m) = \mbox{sign}(\beta) = (-1)^{\lfloor \frac{m}{2} \rfloor}, \]
where $\mbox{sign}(\alpha_j)$ is equal to $1$ if $\alpha_j > 0$ and equal to -1 otherwise. 
\end{lemma}

We now consider a subset $C = \{j_1, \ldots, j_m\} \subseteq [n]$ with $j_1<\cdots<j_m$, such that the corresponding hyperplane has equation $ \alpha_1 z_1 + \ldots \alpha_m z_m - \beta = 0$, as defined above. The previous lemma shows that  the dual variable $y = \alpha/\beta$ satisfies $y_i > 0$ for all $i \in [m]$. We now proceed towards a \textit{gap parity condition} on the subset $C$ under which setting $y=\alpha/\beta$ also satisfies the remaining conditions of \Cref{def:inequalitycell}.



\begin{definition}
(Gaps). For a set $C \subset [n]$, a \textbf{gap} of $C$ refers to an index $i \in [n]\setminus C$. A gap $i$ of $C$ is an \textbf{even gap} if the number of elements in $C$ larger than $i$ is even, and $i$ is an \textbf{odd gap} otherwise.
\end{definition}

\begin{definition}\label{def:gap_parity}
(Facets and Gap Parity).
A subset $C\subseteq[n]$ is called a \textbf{facet} if $|C|=m$ and either: (i) all of its gaps are even; or (ii) all of its gaps are odd.
If all of its gaps are even, then we call $C$ an \textbf{even facet} and define $g(C)=2$.
On the other hand, if all of its gaps are odd, then we call $C$ an \textbf{odd facet} and define $g(C)=1$.
Let $g(C)\in\{1,2\}$ denote the \textbf{gap parity} of a facet $C$, with $g(C)$ being undefined if $C$ is not a facet.
\end{definition}

We now see that every even facet is an inequality cell.

\begin{lemma}\label{lemma:remaining_conditions}

Every even facet $C$ is an inequality cell.
\end{lemma}

The proofs of \cref{lem:hyperplaneequation,lemma:remaining_conditions} require some technical developments on the sub-determinants of $A$ and are deferred to  Appendix~\ref{appendix:proof_lemmas}. The outline of the proof of \cref{lemma:remaining_conditions} is as follows. To show that $C = \{j_1, \ldots, j_m\}$ is an inequality cell, we consider the dual certificate $y = \frac{\alpha}{\beta}$ where $\alpha_1 y_1 + \ldots \alpha_m y_m - \beta = 0$ is the equation of $C$. By \cref{lem:hyperplaneequation},  $\beta$ and $\alpha$ have the same signs, and we know that $\beta \neq 0$ and $\alpha_i \neq 0$ for $i \in [m]$. Therefore, $y_i > 0, \ \ \forall i \in [m].$ For $j \in C$,  
\[ y^\top v'_m(t_j) = \frac{\alpha^\top v'_m(t_j)}{\beta} = \frac{\beta}{\beta} = 1 = c_j.\]
The last step is to show that $ y^\top v'_m(t_j) > c_j$ for $j \not\in C$.

\subsection{Counting the number of $k$-loadouts} \label{sec:pureCounting}

The preceding \Cref{sec:dualCert,sec:reducingToCombinatorial} combine to provide a purely combinatorial lower bound on the number of $k$-loadouts in our construction.
Indeed, \Cref{lem:inequalitycell} shows that a subset $L\subseteq[n]$ with $|L|=k$ is a $k$-loadout as long as $L$ is contained within some inequality cell $C$.
In turn, \Cref{lemma:remaining_conditions} shows that $C$ is an inequality cell as long as it is an even facet.

In this section, we undertake the task of counting the number of $k$-subsets that are contained within at least one even facet, for all $k=1,\ldots,m$. Since such a subset can be contained in different facets, the challenge is not to over-count these subsets.


Following the notation of \cite{eu2010cyclic}, for any subset $L \subseteq [n]$, we associate $L$ with an $(1 \times n)$-array having a star (*) at the $j$th entry if $j \in L$ and a dot (.) otherwise. In such an array, every maximal segment of consecutive stars is called a \textit{block}. A block containing the star at entry $1$ or $n$ is a \textit{border block}, and the other ones are \textit{inner blocks}. The border block containing the star at entry 1 is called the \textit{first border block}, and the one containing the star at $n$ is called the \textit{last border block}. For example, the array associated with $ n = 9$ and subset $L = \{1, 3, 4, 7, 8, 9\}$ is shown in \cref{figure:blocks}, with an inner block $\{3, 4\}$ and border blocks $\{1\}$ and $\{7, 8, 9\}$. A block will be called \textit{even} or \textit{odd} according to the parity of its size. For instance, $\{3, 4\}$ is an even inner block, and $\{7, 8, 9\}$ is an odd last border block.

\begin{figure}[h]
\centering
\includegraphics[scale=1.3]{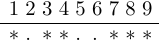}
\caption{The array associated with $n = 9$ and $L = \{1, 3, 4, 7, 8, 9\}$.}
\label{figure:blocks}
\end{figure}

For $1 \le k \le m$ and $0 \le s \le k$, let $A(n, k, s)$ be the set of $(1\times n)$-arrays with $k$ stars and $s$ odd inner blocks. We further define $A^{\mathrm{odd}}(n, k, s)$ (resp. $A^{\mathrm{even}}(n, k, s)$) to be the set of $(1\times n)$-arrays with $k$ stars, $s$ odd inner blocks and an odd (resp. even) last border block, such that 
\begin{equation*}
|A^{\mathrm{odd}}(n, k, s)| + |A^{\mathrm{even}}(n, k, s)| = |A(n, k, s)|.
\end{equation*}
Note that the last border block can be empty (occurring when there is a (.) in position $n$) and such a block is considered even.

We first show that the set of arrays corresponding to facets is $A(n,m,0)$, and that the set of arrays of $k$-subsets that are included in a facet contains $\cup_{s=0}^{m-k} A(n,k,s)$.

\begin{lemma}\label{lemma:face_array}
The set of facets (as per \Cref{def:gap_parity}) corresponds to the set of arrays with $m$ stars and no odd inner blocks.  In other words, the set of facets is equal to $A(n,m,0)$. The set of even facets is $A^{\mathrm{even}}(n,m,0)$ and the set of odd facets is $A^{\mathrm{odd}}(n,m,0)$. Furthermore, for $1 \le k \le m-1$, every $k$-subset in $\cup_{s=0}^{m-k} A(n,k,s)$ is contained in a facet.
\end{lemma}
\proof 
Let $C$ be an even facet. Let $j \not\in C$ be the greatest gap in $C$. By definition of an even facet, the number of indices in $C$ larger than $j$ is even. Because $j$ is the greatest gap, the elements in $C$ larger than $j$  constitute the last border block. Therefore the last border block of $C$ is even. Now, consider the rightmost inner block of $C$, if this block is odd, then there exists an odd gap of $C$, which contradicts the fact that $C$ is an even. By considering the remaining inner blocks from right to left, we can see that if any of these blocks is odd, then $C$ will have an odd gap, contradicting the fact that it's an even facet. Therefore, the array of $C$ has no odd inner blocks and an even last border block. Similarly, we can show that if $C$ is an odd facet, the array of $C$ has no odd inner blocks and an odd last border block. This shows that every facet has no odd inner blocks.

Conversely, consider a subset $C$ whose array is $A(n,m,0)$. This implies that $|C| = m$ and $C$ has no odd inner blocks. One can see that since all the inner blocks are even, then all the gaps of $C$ have the same parity as the last border block of $C$. Therefore, $C$ is a facet.

Consider $1 \le k \le m-1$ and $L \subseteq [n]$ such that $|L| = k$ and the array of $L$ is in $\cup_{s=0}^{m-k} A(n,k,s)$. We show that we can add $m-k$ stars to the array of $L$ to get rid of all the odd inner blocks. This implies that $L$ is included into a facet. Since $L$ has less than $s \leq m-k$ odd inner blocks, we can add $s$ stars to the right of every odd inner block. This ensures that there is no odd inner block. We then add $m-k-s$ stars to the right of the first border block of the array. 
\endproof

Recall that we are interested in the $k$-subsets that are included in even facets. The next lemma presents a sufficient condition for a $k$-subset to be included in both an even and an odd facet.

\begin{lemma}\label{lemma:facets_intersection}
For $1\le  k \leq m-1$, any $k$-subset with strictly less than $m-k$ odd inner blocks is included in both an even facet and an odd facet.
\end{lemma}
\proof 
We present the proof for the case of odd facets. The other case is argued similarly. Let $L \subseteq [n]$ such that $|L| = k$ and the corresponding array has $s$ odd inner blocks, with $s < m-k$. We will augment the array corresponding to $L$ to become an array corresponding to an odd facet by adding $m-k$ stars. This will prove that $L$ is contained in a facet. Consider the following procedure:
\begin{enumerate}
    \item Go over all the odd inner blocks of $L$ from left to right.
    \item For every odd inner block, add a star to the right of the block.
    \item If the last border block is odd, add the remaining stars to the right of the first border block. Otherwise, add one star to the left of the last border block and the remaining stars to the right of the first border block.
\end{enumerate}
After step 2 of the procedure above, every inner block has been transformed to either an even inner block or to be a part of the last border block. In fact, after adding a star to the right of an inner block, we distinguish the following three cases:
1) The added star does not connect the block to any other block. In this case, the block becomes even.
2) The added star connects the odd inner block to an even inner block. In this case, the new block is even. 3) The added star connects the odd inner block to an odd inner block. In this case, we keep adding a star to the right. 

If the last border block is odd, then we can add the remaining $m-k -s$ stars to the right of the first border block. Since $m < n$, we can always do so without affecting the last border block. If the last border block is even, then we add one star to the left of this border block. If the added star doesn't connect this border block to any other block, then we are done. If the added star connects this border block to another block $\alpha$, then $\alpha$ is an even block and, therefore, the last border block changes parity because it now has 1 + $|\alpha|$ additional stars, and 1 + $|\alpha|$ is odd. 
\endproof

The next lemma shows that when a $k$-subset has $m-k$ odd inner blocks ($1 \le k \le m$), then it is included in a facet with gap parity opposite to $m$ if the last border block's parity is opposite to $m$.

\begin{lemma}\label{lemma:mk_odd}
Let $L \subseteq [n]$ such that $|L| = k \in [n]$. If the array of $L$ is in $A^{\mathrm{even}}(n,k,m-k)$, then $L$ is included in an even facet. Similarly, if the array of $L$ is in $A^{\mathrm{odd}}(n,k,m-k)$, then $L$ is included in an odd facet.
\end{lemma}
\proof 
We prove the result in the even case. The odd case is argued similarly. Let $\alpha \in A^{\mathrm{even}}(n,k,m-k)$ be the array of $L$. By adding one star to every inner odd block in $\alpha$ (exactly $m-k$ stars added), we ensure that the resulting array has 0 inner odd blocks and an even last border block. The resulting array corresponds then to an even facet and $L$ is included in an even facet. 
\endproof

\begin{lemma}
\label{cor:small_dimension}
Recall that $|\L^k(A,c)|$ is the number of $k$-loadouts in our construction. We have for $1 \le k \le m$
\begin{align} \label{eqn:1234}
|\L^k(A,c)| \geq  \sum\limits_{s = 0}^{m-k-1} |A(n, k, s)| + |A^{\mathrm{even}}(n, k, m-k)|.
\end{align}

\end{lemma}

\proof  Let $1 \le k \le m$. We showed in  \cref{lemma:remaining_conditions}, that every subset of an even facet is a loadout of our construction. Combining \cref{lemma:face_array} and \cref{lemma:facets_intersection}, we show that any $k$-subset with strictly less than $m-k$ odd inner blocks is included in an even facet. \cref{lemma:mk_odd} shows that any $k$-subset with exactly $m-k$ odd inner blocks and an even last border block is included in an even facet. Therefore $|\L^k(A,c)| \geq  \sum\limits_{s = 0}^{m-k-1} |A(n, k, s)| + |A^{\mathrm{even}}(n, k, m-k)|. $
\endproof

In the rest of the section, we show that for $k < \lfloor m/2\rfloor$, we have $|\L^k(A,c)| \geq \sum_{s=0}^{m-k} |A(n,k,s)|,$ and for $k \geq \lfloor m/2\rfloor$, we have $|\L^k(A,c)| \geq \big(\sum_{s=0}^{m-k} |A(n,k,s)|\big)/2$. We start with small values of $k$.

\begin{corollary}\label{lemma:small_k}
When $k <  m/2$, we have $|\L^k(A,c)| \geq \sum_{s=0}^{m-k} |A(n,k,s)|.$
\end{corollary}
\proof 
We first observe that when $k <  m/2$, then $m-k > k$ and therefore $|A(n, k, m-k)|=0$. \cref{cor:small_dimension} implies then that $|\L^k(A,c)| \geq \sum_{s=0}^{m-k} |A(n,k,s)|$. 
\endproof

We ovserve that $|A^{\mathrm{odd}}(n, k, m-k)| \leq |A^{\mathrm{even}}(n, k, m-k)|$.

\begin{lemma}\label{lemma:Aodd_even} For $1 \le k \le m$, we have $|A^{\mathrm{odd}}(n, k, m-k)| \leq |A^{\mathrm{even}}(n, k, m-k)|$.
\end{lemma}

\proof 
 Let $\alpha \in A^{\mathrm{odd}}(n, k, m-k)$. We can transform $\alpha $ to an array from $|A^{\mathrm{even}}(n, k, m-k)|$ as follows: we take the first star to the left of the last border block and add it to the right of the first border block and translate all inner blocks by 1 to the right. The resulting array is in $|A^{\mathrm{even}}(n, k, m-k)|$. One can easily see that both operations are injective. Therefore, $|A^{\mathrm{even}}(n, k, m-k)| \leq |A^{\mathrm{odd}}(n, k, m-k)|$. 
\endproof

\begin{corollary}\label{cor:m_odd}
For $1 \le k \le m$, we have $|\L^k(A,c)| \geq \big(\sum_{s=0}^{m-k} |A(n,k,s)|\big)/2.$ 
\end{corollary}
\proof 
By \cref{cor:small_dimension},
\begin{align*}
    |\L^k(A,c)| & \geq  \sum\limits_{s = 0}^{m-k-1} |A(n, k, s)| + |A^{\mathrm{even}}(n, k, m-k)|\\
    & \geq \sum\limits_{s = 0}^{m-k-1} |A(n, k, s)| + \frac{|A^{\mathrm{even}}(n, k, m-k)|+|A^{\mathrm{odd}}(n, k, m-k)|}{2}\\
    & = \sum\limits_{s = 0}^{m-k-1} |A(n, k, s)| + \frac{|A(n, k, m-k)|}{2} \geq \frac{\sum\limits_{s = 0}^{m-k} |A(n, k, s)|}{2},
\end{align*}
where the second inequality uses the inequality $|A^{\mathrm{even}}(n, k, m-k)| \geq |A^{\mathrm{odd}}(n, k, m-k)|$ provided by \cref{lemma:Aodd_even}. 
\endproof

Finally, let us prove that
\begin{equation}\label{eq:af}
    \sum_{s=0}^{m-k} |A(n,k,s)| = f_{k-1}\big(\C(n,m)\big),\quad\text{for any }k\in\{1,\ldots,m\}.
\end{equation}
In conjunction with our previous lemmas, the formula \cref{eq:af} would imply that $|\L^k(A,c)|\geq\tfrac{1}{2}f_{k-1}\big(\C(n,m)\big)$, for any $k\in\{1,\ldots,m\}$. Consequently, since $|\L^k(A,c)\leq f_{k-1}\big(\C(n+1,m)\big)$, due to \cref{thm:upperbound}, and since $\lim_{n\rightarrow \infty} \frac{f_{k-1}(\C(n, m))}{f_{k-1}(\C(n+1, m))} = 1$, due to \cref{lemma:asymptotic_fk}, our construction is asymptotically a $1/2$-approximation.

 To prove \cref{eq:af}, we invoke the following criterion for determining the faces of $\C(n, m)$.
    \begin{theorem}[\cite{shephard1968theorem}]\label{thm:shephard}
      For $1 \leq k \leq m$ , a subset $L \subseteq [n]$ is the set of vertices of a $(k - 1)$-dimensional face of $\C(n, m)$ if and only if $|L|=k$ and its associated array contains at most $m - k$ odd inner blocks.
    \end{theorem}
    
An immediate consequence of the above theorem is that, for $1 \le k \le m$, 
$
 f_{k-1}\big(\C(n,m)\big) = \sum_{s=0}^{m-k} |A(n,k,s)|.
$
We are now ready to complete the proof of \Cref{thm:lowerbound}.

\begin{proof}[Proof of \cref{thm:lowerbound}]
We distinguish two cases. When $k <  m/2 $, \cref{lemma:small_k} implies that $|\L^k(A,c)| \geq \sum_{s=0}^{m-k} |A(n,k,s)| = f_{k-1}\big(\C(n,m)\big)$. When $k\geq m/2$, by \cref{cor:m_odd}, we have $|\L^k(A,c)| \geq \big(\sum_{s=0}^{m-k} |A(n,k,s)|\big)/2 = f_{k-1}\big(\C(n, m)\big)/2$. 
%
\end{proof}

\section{Conclusion}\label{s:conclusion}

We study the novel problem of diversity maximization, motivated naturally by the video game design context where designing for diversity is one of its core design philosophies. We model this diversity optimization problem as a parametric linear programming problem where we are interested in the diversity of supports of optimal solutions. Using this model, we establish upper bounds and construct game designs that match this upper bound asymptotically.

Returning briefly to the video game motivation, our analysis in this paper provides practical guidance for video game developers who are interested to design their game with an explicit progression of loadouts in mind. Suppose the game designer would prefer players to follow a weapon loadout pattern $L_1, L_2, L_3, \dots$, as the game proceeds. Provided this sequence of loadouts appears as supports of optimal solutions to $L(A,c,b)$ for some choice of $A$, $b$, and $c$ (which would arise for an optimizing player) our theory could tell the game designer how to progress the resource vector $b$ to make the loadout sequence $L_1, L_2, L_3, \dots$ optimal. This simply requires choosing $b$ in the interior of the conic combination of the collection of columns of $A$ corresponding to a given loadout.  

To our knowledge, ours is the first paper to systematically study the question of ``diversity maximization'' as we have defined it here. The goal here is ``diverse-in diverse-out'', if two players have diverse resources (different right-hand resource vectors), they will optimally play different strategies. 

\cref{fig:construction} shows that although our construction and upper bound have promising asymptotic properties, there is certainly scope for improving both. Another counting argument could yield a tight lower bound to the construction, and it may be possible to leverage new ideas to strengthen the upper bound. 

We believe there could be other applications for ``diverse-in diverse-out'' optimization problems. Consider, for example, a diet problem where a variety of ingredients are used in the making of meals, depending on different availability in resources. We leave this exploration for future work. There are also natural extensions to our model and analysis that could be pursued. For instance, we have studied the linear programming version of the problem. An obvious next step is the integer linear setting, which also arises naturally in the design of games. For example, \cite{supertank} proposed a $\{0,1\}$-formulation of the game SuperTank.

Just as in our analysis of the linear program, a deep understanding of the parametric nature of the integer optimization problems is necessary to proceed in the integer setting. The work \cite{sturmfels1997variation} introduces a theory of reduced Gr\"obner bases of toric ideals that play a role analogous to triangulations of cones. We leave this as an interesting direction for further investigation to build on this parametric theory.

Of course, an even more compelling extension would involve \emph{mixed}-integer decision sets. This will require a deep appreciation of parametric mixed-integer linear programming, a topic that remains of keen interest in the integer programming community (see, for instance,  \cite{eisenbrand2008parametric,gribanov2022delta}). In this case, the integer programming theory necessary to study the diversity maximization problem is still being developed.

Yet another direction is to consider multiple objectives for the player. In our setting, we have assumed a single meaningful objective for the player, such as maximizing the damage of a loadout of weapons. In some games, other objectives may be possible, including the cosmetics of weapons or balancing a mix of tools with offensive and defensive attributes. There exists theory on parametric multi-objective optimization that could serve as a starting point (see, for instance \cite{tanino1988sensitivity}). 

\bmhead{Acknowledgments}
Thanks to Jesus De Loera for some enlightening guidance on triangulations and to Xiao Lei for some early useful feedback. The second author would like to thank combinatorialist Steven Karp for insightful discussions surrounding the cyclic polytope. The third author would like to thank video game designer Paul Tozour for introducing him to the diversity optimization problem in video game design. The third author's research is supported by a Discovery Grant of the Natural Science and Engineering Research Council of Canada and an exploratory research grant from the UBC Sauder School of Business.  The authors thank the anonymous reviewers of Math Programming for extremely detailed feedback that substantially improved the final version.

\bmhead{Compliance with Ethical Standards}

The authors have no financial or non-financial interests that are directly or indirectly related to this work.

\bibliography{sn-bibliography}


\begin{thebibliography}{30}
\ifx \bisbn   \undefined \def \bisbn  #1{ISBN #1}\fi
\ifx \binits  \undefined \def \binits#1{#1}\fi
\ifx \bauthor  \undefined \def \bauthor#1{#1}\fi
\ifx \batitle  \undefined \def \batitle#1{#1}\fi
\ifx \bjtitle  \undefined \def \bjtitle#1{#1}\fi
\ifx \bvolume  \undefined \def \bvolume#1{\textbf{#1}}\fi
\ifx \byear  \undefined \def \byear#1{#1}\fi
\ifx \bissue  \undefined \def \bissue#1{#1}\fi
\ifx \bfpage  \undefined \def \bfpage#1{#1}\fi
\ifx \blpage  \undefined \def \blpage #1{#1}\fi
\ifx \burl  \undefined \def \burl#1{\textsf{#1}}\fi
\ifx \doiurl  \undefined \def \doiurl#1{\url{https://doi.org/#1}}\fi
\ifx \betal  \undefined \def \betal{\textit{et al.}}\fi
\ifx \binstitute  \undefined \def \binstitute#1{#1}\fi
\ifx \binstitutionaled  \undefined \def \binstitutionaled#1{#1}\fi
\ifx \bctitle  \undefined \def \bctitle#1{#1}\fi
\ifx \beditor  \undefined \def \beditor#1{#1}\fi
\ifx \bpublisher  \undefined \def \bpublisher#1{#1}\fi
\ifx \bbtitle  \undefined \def \bbtitle#1{#1}\fi
\ifx \bedition  \undefined \def \bedition#1{#1}\fi
\ifx \bseriesno  \undefined \def \bseriesno#1{#1}\fi
\ifx \blocation  \undefined \def \blocation#1{#1}\fi
\ifx \bsertitle  \undefined \def \bsertitle#1{#1}\fi
\ifx \bsnm \undefined \def \bsnm#1{#1}\fi
\ifx \bsuffix \undefined \def \bsuffix#1{#1}\fi
\ifx \bparticle \undefined \def \bparticle#1{#1}\fi
\ifx \barticle \undefined \def \barticle#1{#1}\fi
\bibcommenthead
\ifx \bconfdate \undefined \def \bconfdate #1{#1}\fi
\ifx \botherref \undefined \def \botherref #1{#1}\fi
\ifx \url \undefined \def \url#1{\textsf{#1}}\fi
\ifx \bchapter \undefined \def \bchapter#1{#1}\fi
\ifx \bbook \undefined \def \bbook#1{#1}\fi
\ifx \bcomment \undefined \def \bcomment#1{#1}\fi
\ifx \oauthor \undefined \def \oauthor#1{#1}\fi
\ifx \citeauthoryear \undefined \def \citeauthoryear#1{#1}\fi
\ifx \endbibitem  \undefined \def \endbibitem {}\fi
\ifx \bconflocation  \undefined \def \bconflocation#1{#1}\fi
\ifx \arxivurl  \undefined \def \arxivurl#1{\textsf{#1}}\fi
\csname PreBibitemsHook\endcsname

\bibitem[\protect\citeauthoryear{Hanguir et~al.}{2023}]{hanguir2023designing}
\begin{bchapter}
\bauthor{\bsnm{Hanguir}, \binits{O.}},
\bauthor{\bsnm{Ma}, \binits{W.}},
\bauthor{\bsnm{Ryan}, \binits{C.T.}}:
\bctitle{Designing optimization problems with diverse solutions}.
In: \bbtitle{International Conference on Integer Programming and Combinatorial Optimization},
pp. \bfpage{172}--\blpage{186}
(\byear{2023}).
\bcomment{Springer}
\end{bchapter}
\endbibitem

\bibitem[\protect\citeauthoryear{Schoenau-Fog et~al.}{2011}]{schoenau2011player}
\begin{bchapter}
\bauthor{\bsnm{Schoenau-Fog}, \binits{H.}}, \betal:
\bctitle{The player engagement process --- {A}n exploration of continuation desire in digital games.}
In: \bbtitle{DiGRA Conference}
(\byear{2011})
\end{bchapter}
\endbibitem

\bibitem[\protect\citeauthoryear{Knight}{2015}]{smite}
\begin{botherref}
\oauthor{\bsnm{Knight}, \binits{V.}}:
Maths behind a {MOBA}: Using Linear Programming to model and solve a problem.
\url{https://vknight.org/Computing_for_mathematics/Assessment/IndividualCoursework/PastCourseWorks/2015-2016/knight2015-2016.pdf}
(2015)
\end{botherref}
\endbibitem

\bibitem[\protect\citeauthoryear{Knight}{2014}]{clashofclans}
\begin{botherref}
\oauthor{\bsnm{Knight}, \binits{V.}}:
Wizards, Giants, Linear Programming and Sage.
\url{http://drvinceknight.blogspot.com/2014/05/wizards-giants-linear-programming-and.html}
(2014)
\end{botherref}
\endbibitem

\bibitem[\protect\citeauthoryear{Tozour}{2013}]{supertank}
\begin{botherref}
\oauthor{\bsnm{Tozour}, \binits{P.}}:
Decision Modeling and Optimization in Game Design, Part 1: Introduction.
\url{https://www.gamasutra.com/blogs/PaulTozour/20130707/195718/Decision_Modeling_and_Optimization_in_Game_Design_Part_1_Introduction.php}
(2013)
\end{botherref}
\endbibitem

\bibitem[\protect\citeauthoryear{Saaty and Gass}{1954}]{saaty1954parametric}
\begin{barticle}
\bauthor{\bsnm{Saaty}, \binits{T.}},
\bauthor{\bsnm{Gass}, \binits{S.}}:
\batitle{Parametric objective function (part 1)}.
\bjtitle{Journal of the Operations Research Society of America}
\bvolume{2}(\bissue{3}),
\bfpage{316}--\blpage{319}
(\byear{1954})
\end{barticle}
\endbibitem

\bibitem[\protect\citeauthoryear{Mills}{1956}]{mills1956marginal}
\begin{bchapter}
\bauthor{\bsnm{Mills}, \binits{H.}}:
\bctitle{Marginal values of matrix games and linear programs}.
In: \beditor{\bsnm{Kuhn}, \binits{H.W.}},
\beditor{\bsnm{Tucker}, \binits{A.W.}} (eds.)
\bbtitle{Linear Inequalities and Related Systems},
pp. \bfpage{183}--\blpage{194}.
\bpublisher{Princeton University Press},
\blocation{Princeton}
(\byear{1956})
\end{bchapter}
\endbibitem

\bibitem[\protect\citeauthoryear{Williams}{1963}]{williams1963marginal}
\begin{barticle}
\bauthor{\bsnm{Williams}, \binits{A.}}:
\batitle{Marginal values in linear programming}.
\bjtitle{Journal of the Society for Industrial and Applied Mathematics}
\bvolume{11}(\bissue{1}),
\bfpage{82}--\blpage{94}
(\byear{1963})
\end{barticle}
\endbibitem

\bibitem[\protect\citeauthoryear{Walkup and Wets}{1969}]{walkup1969lifting}
\begin{barticle}
\bauthor{\bsnm{Walkup}, \binits{D.}},
\bauthor{\bsnm{Wets}, \binits{R.}}:
\batitle{Lifting projections of convex polyhedra}.
\bjtitle{Pacific Journal of Mathematics}
\bvolume{28}(\bissue{2}),
\bfpage{465}--\blpage{475}
(\byear{1969})
\end{barticle}
\endbibitem

\bibitem[\protect\citeauthoryear{Parre{\~n}o et~al.}{2021}]{parreno2021measuring}
\begin{barticle}
\bauthor{\bsnm{Parre{\~n}o}, \binits{F.}},
\bauthor{\bsnm{{\'A}lvarez-Vald{\'e}s}, \binits{R.}},
\bauthor{\bsnm{Mart{\'\i}}, \binits{R.}}:
\batitle{Measuring diversity. a review and an empirical analysis}.
\bjtitle{European Journal of Operational Research}
\bvolume{289}(\bissue{2}),
\bfpage{515}--\blpage{532}
(\byear{2021})
\end{barticle}
\endbibitem

\bibitem[\protect\citeauthoryear{Turner et~al.}{2011}]{turner2011or}
\begin{barticle}
\bauthor{\bsnm{Turner}, \binits{J.}},
\bauthor{\bsnm{Scheller-Wolf}, \binits{A.}},
\bauthor{\bsnm{Tayur}, \binits{S.}}:
\batitle{Scheduling of dynamic in-game advertising}.
\bjtitle{Operations Research}
\bvolume{59}(\bissue{1}),
\bfpage{1}--\blpage{16}
(\byear{2011})
\end{barticle}
\endbibitem

\bibitem[\protect\citeauthoryear{Guo et~al.}{2019}]{guo2019economic}
\begin{barticle}
\bauthor{\bsnm{Guo}, \binits{H.}},
\bauthor{\bsnm{Zhao}, \binits{X.}},
\bauthor{\bsnm{Hao}, \binits{L.}},
\bauthor{\bsnm{Liu}, \binits{D.}}:
\batitle{Economic analysis of reward advertising}.
\bjtitle{Production and Operations Management}
\bvolume{28}(\bissue{10}),
\bfpage{2413}--\blpage{2430}
(\byear{2019})
\end{barticle}
\endbibitem

\bibitem[\protect\citeauthoryear{Sheng et~al.}{2020}]{sheng2020incentivized}
\begin{botherref}
\oauthor{\bsnm{Sheng}, \binits{L.}},
\oauthor{\bsnm{Ryan}, \binits{C.T.}},
\oauthor{\bsnm{Nagarajan}, \binits{M.}},
\oauthor{\bsnm{Cheng}, \binits{Y.}},
\oauthor{\bsnm{Tong}, \binits{C.}}:
Incentivized actions in freemium games.
Manufacturing \& Service Operations Management
(2020)
\end{botherref}
\endbibitem

\bibitem[\protect\citeauthoryear{Guo et~al.}{2019}]{guo2019selling}
\begin{barticle}
\bauthor{\bsnm{Guo}, \binits{H.}},
\bauthor{\bsnm{Hao}, \binits{L.}},
\bauthor{\bsnm{Mukhopadhyay}, \binits{T.}},
\bauthor{\bsnm{Sun}, \binits{D.}}:
\batitle{Selling virtual currency in digital games: Implications for gameplay and social welfare}.
\bjtitle{Information Systems Research}
\bvolume{30}(\bissue{2}),
\bfpage{430}--\blpage{446}
(\byear{2019})
\end{barticle}
\endbibitem

\bibitem[\protect\citeauthoryear{Chen et~al.}{2020}]{chen2020loot}
\begin{barticle}
\bauthor{\bsnm{Chen}, \binits{N.}},
\bauthor{\bsnm{Elmachtoub}, \binits{A.N.}},
\bauthor{\bsnm{Hamilton}, \binits{M.L.}},
\bauthor{\bsnm{Lei}, \binits{X.}}:
\batitle{Loot box pricing and design}.
\bjtitle{Management Science}
\bvolume{67}(\bissue{8}),
\bfpage{4809}--\blpage{4825}
(\byear{2020})
\end{barticle}
\endbibitem

\bibitem[\protect\citeauthoryear{Ryan et~al.}{2020}]{ryan2020selling}
\begin{botherref}
\oauthor{\bsnm{Ryan}, \binits{C.T.}},
\oauthor{\bsnm{Sheng}, \binits{L.}},
\oauthor{\bsnm{Zhao}, \binits{X.}}:
Selling enhanced attempts.
Available at SSRN 3751523
(2020)
\end{botherref}
\endbibitem

\bibitem[\protect\citeauthoryear{Chen et~al.}{2017}]{chen2017eomm}
\begin{bchapter}
\bauthor{\bsnm{Chen}, \binits{Z.}},
\bauthor{\bsnm{Xue}, \binits{S.}},
\bauthor{\bsnm{Kolen}, \binits{J.}},
\bauthor{\bsnm{Aghdaie}, \binits{N.}},
\bauthor{\bsnm{Zaman}, \binits{K.A.}},
\bauthor{\bsnm{Sun}, \binits{Y.}},
\bauthor{\bsnm{Seif~El-Nasr}, \binits{M.}}:
\bctitle{{EOMM}: An engagement optimized matchmaking framework}.
In: \bbtitle{Proceedings of the 26th International Conference on World Wide Web},
pp. \bfpage{1143}--\blpage{1150}
(\byear{2017})
\end{bchapter}
\endbibitem

\bibitem[\protect\citeauthoryear{Huang et~al.}{2019}]{huang2019level}
\begin{barticle}
\bauthor{\bsnm{Huang}, \binits{Y.}},
\bauthor{\bsnm{Jasin}, \binits{S.}},
\bauthor{\bsnm{Manchanda}, \binits{P.}}:
\batitle{“{L}evel up”: Leveraging skill and engagement to maximize player game-play in online video games}.
\bjtitle{Information Systems Research}
\bvolume{30}(\bissue{3}),
\bfpage{927}--\blpage{947}
(\byear{2019})
\end{barticle}
\endbibitem

\bibitem[\protect\citeauthoryear{McMullen}{1970}]{mcmullen1970maximum}
\begin{barticle}
\bauthor{\bsnm{McMullen}, \binits{P.}}:
\batitle{The maximum numbers of faces of a convex polytope}.
\bjtitle{Mathematika}
\bvolume{17}(\bissue{2}),
\bfpage{179}--\blpage{184}
(\byear{1970})
\end{barticle}
\endbibitem

\bibitem[\protect\citeauthoryear{De~Loera et~al.}{2010}]{de2010triangulations}
\begin{bbook}
\bauthor{\bsnm{De~Loera}, \binits{J.}},
\bauthor{\bsnm{Rambau}, \binits{J.}},
\bauthor{\bsnm{Santos}, \binits{F.}}:
\bbtitle{Triangulations: Structures for Algorithms and Applications}
vol. \bseriesno{25}.
\bpublisher{Springer},
\blocation{Heidelberg}
(\byear{2010})
\end{bbook}
\endbibitem

\bibitem[\protect\citeauthoryear{Sturmfels and Thomas}{1997}]{sturmfels1997variation}
\begin{barticle}
\bauthor{\bsnm{Sturmfels}, \binits{B.}},
\bauthor{\bsnm{Thomas}, \binits{R.R.}}:
\batitle{Variation of cost functions in integer programming}.
\bjtitle{Mathematical Programming}
\bvolume{77}(\bissue{2}),
\bfpage{357}--\blpage{387}
(\byear{1997})
\end{barticle}
\endbibitem

\bibitem[\protect\citeauthoryear{Soltan}{2019}]{soltan2019lectures}
\begin{bbook}
\bauthor{\bsnm{Soltan}, \binits{V.}}:
\bbtitle{Lectures on Convex Sets}
vol. \bseriesno{986}.
\bpublisher{World Scientific},
\blocation{Hackensack, NJ}
(\byear{2019})
\end{bbook}
\endbibitem

\bibitem[\protect\citeauthoryear{Beck and Robins}{2007}]{beckcomputing}
\begin{bbook}
\bauthor{\bsnm{Beck}, \binits{M.}},
\bauthor{\bsnm{Robins}, \binits{S.}}:
\bbtitle{Computing the Continuous Discretely: Integer-Point Enumeration in Polyhedra}.
\bpublisher{Springer},
\blocation{Heidelberg}
(\byear{2007})
\end{bbook}
\endbibitem

\bibitem[\protect\citeauthoryear{Kalai}{1987}]{kalai1987rigidity}
\begin{barticle}
\bauthor{\bsnm{Kalai}, \binits{G.}}:
\batitle{Rigidity and the lower bound theorem 1}.
\bjtitle{Inventiones Mathematicae}
\bvolume{88}(\bissue{1}),
\bfpage{125}--\blpage{151}
(\byear{1987})
\end{barticle}
\endbibitem

\bibitem[\protect\citeauthoryear{Eu et~al.}{2010}]{eu2010cyclic}
\begin{barticle}
\bauthor{\bsnm{Eu}, \binits{S.-P.}},
\bauthor{\bsnm{Fu}, \binits{T.-S.}},
\bauthor{\bsnm{Pan}, \binits{Y.-J.}}:
\batitle{The cyclic sieving phenomenon for faces of cyclic polytopes}.
\bjtitle{The Electronic Journal of Combinatorics}
\bvolume{17}(\bissue{1}),
\bfpage{47}
(\byear{2010})
\end{barticle}
\endbibitem

\bibitem[\protect\citeauthoryear{Shephard}{1968}]{shephard1968theorem}
\begin{barticle}
\bauthor{\bsnm{Shephard}, \binits{G.C.}}:
\batitle{A theorem on cyclic polytopes}.
\bjtitle{Israel Journal of Mathematics}
\bvolume{6}(\bissue{4}),
\bfpage{368}--\blpage{372}
(\byear{1968})
\end{barticle}
\endbibitem

\bibitem[\protect\citeauthoryear{Eisenbrand and Shmonin}{2008}]{eisenbrand2008parametric}
\begin{barticle}
\bauthor{\bsnm{Eisenbrand}, \binits{F.}},
\bauthor{\bsnm{Shmonin}, \binits{G.}}:
\batitle{Parametric integer programming in fixed dimension}.
\bjtitle{Mathematics of Operations Research}
\bvolume{33}(\bissue{4}),
\bfpage{839}--\blpage{850}
(\byear{2008})
\end{barticle}
\endbibitem

\bibitem[\protect\citeauthoryear{Gribanov et~al.}{2022}]{gribanov2022delta}
\begin{botherref}
\oauthor{\bsnm{Gribanov}, \binits{D.}},
\oauthor{\bsnm{Shumilov}, \binits{I.}},
\oauthor{\bsnm{Malyshev}, \binits{D.}},
\oauthor{\bsnm{Pardalos}, \binits{P.}}:
On $\delta$-modular integer linear problems in the canonical form and equivalent problems.
Journal of Global Optimization,
1--61
(2022)
\end{botherref}
\endbibitem

\bibitem[\protect\citeauthoryear{Tanino}{1988}]{tanino1988sensitivity}
\begin{barticle}
\bauthor{\bsnm{Tanino}, \binits{T.}}:
\batitle{Sensitivity analysis in multiobjective optimization}.
\bjtitle{Journal of Optimization Theory and Applications}
\bvolume{56}(\bissue{3}),
\bfpage{479}--\blpage{499}
(\byear{1988})
\end{barticle}
\endbibitem

\bibitem[\protect\citeauthoryear{Fekete and P{\'o}lya}{1912}]{fekete1912problem}
\begin{barticle}
\bauthor{\bsnm{Fekete}, \binits{M.}},
\bauthor{\bsnm{P{\'o}lya}, \binits{G.}}:
\batitle{{\"U}ber ein problem von laguerre}.
\bjtitle{Rendiconti del Circolo Matematico di Palermo (1884-1940)}
\bvolume{34}(\bissue{1}),
\bfpage{89}--\blpage{120}
(\byear{1912})
\end{barticle}
\endbibitem

\end{thebibliography}

\newpage

\begin{appendices}

\pagenumbering{arabic}
\renewcommand*{\thepage}{A.\arabic{page}}
\renewcommand{\thelemma}{A.\arabic{lemma}}
\renewcommand{\thesection}{A.\arabic{section}}
\renewcommand{\theproposition}{A.\arabic{proposition}}
\renewcommand{\thecorollary}{A.\arabic{corollary}}
\renewcommand{\theequation}{A.\arabic{equation}}
\renewcommand{\theremark}{A.\arabic{remark}}
\setcounter{lemma}{0}
\setcounter{section}{0}
\setcounter{proposition}{0}
\setcounter{equation}{0}
\setcounter{corollary}{0}

\section{Omitted proofs}

\subsection{Properties of a cone triangulation}\label{appendixA}

\begin{proof}[Proof of \cref{lemma:3properties}]
We present a geometric proof. Recall that we can think of the subdivision $\Delta_c(A)$ as follows: take the cost vector $c$,  and use it to lift the columns of $A$ to $\R^{m+1}$ then look at the projection of the upper faces (those faces you would see if you ``look from above'') of the lifted point set. The projection of every one of these faces is a cell of $\Delta_c(A)$. 
\begin{itemize}
\setlength{\itemindent}{1.5em}
    \item[(CP)] Let $C$ be a cell of $\Delta_c(A)$ and $F$ be a face of $C$. Since $C$ is an upper face, every face of $C$ can also be ``seen from above'' and is therefore a cell of $\Delta_c(A)$. 
    \item[(UP)] Let $x \in \cone(\{1,\ldots,n\})$. The
intersection of $x \times \R$ with the convex hull of the elevated columns is a vertical segment from a bottom
point $x_1$ to a top point $x_2$. Let $F$ be any proper face of this convex hull that contains $x_2$, which exists since $x_2$ is in the boundary. $F$ is an upper face and its projection is a cell in $\Delta_c(A)$ that contains $x$.
    \item[(IP)] The intersection property follows from the intersection property of the faces of the elevated polytope. \qedhere
\end{itemize}
\end{proof}

\subsection{Maximizing loadouts is trivial when $n \leq m$}
\begin{lemma}\label{lemma:n_leq_m}
Suppose $n \leq m$. In this case, a trivial design is optimal. By setting $A = I_n$ to be the identity matrix of size $n$, and $c = (1,\ldots,1)$, we have that for $k \in [1,n]$, every one of the $\binom{n}{k}$ subsets is a loadout.
\end{lemma}

\proof 
Consider $1 \leq k \leq n$, and $L \subseteq [n]$ such that $|L| = k$. Consider the resource vector $b \in \R^m$ such that $b_j = 1$ if $j \in L$ and $b_j = 0$ otherwise. In this case, the linear program $L(A,c,b)$ can be written as 
\begin{alignat*}{3}
 \text{maximize} & \sum\limits_{j \in L} x_j& \\
  \text{s.t.} \quad& x_j \leq 1 \mbox{ for } j \in L\\
  &  x_j = 0 \mbox{ for } j \not\in L\\
                 &  x_j \geq 0 \mbox{ for } j \in [n].
\end{alignat*}
The unique optimal solution to $L(A,c,b)$ in this case is such that $x_j = 1 \mbox{ for } j \in L$ and  $x_j = 0 \mbox{ for } j \not\in L$. Therefore $L$ is a loadout, and every subset of size $k$ is a loadout in the design $(A,c)$. 
\endproof

\subsection{Proof of \cref{lemma:asymptotic_fk}}\label{appx:asymptotic_fk}
\begin{lemma}\label{lemma:asymptotic_fk}
For $1 \leq k \leq m$
\begin{align*}
\lim_{n\to\infty}\frac{f_{k-1}(\C(n,m))}{f_{k-1}(\C(n+1,m))}=1.
\end{align*}
\end{lemma}
\proof 
We prove the lemma when $m$ is even. The other case is argued symmetrically. When $m$ is even, the number of faces $f_{k-1}(\C(n,m))$ can be written as follows \citep{eu2010cyclic},
\[ f_{k-1}(\C(n,m)) = \sum\limits_{j = 1}^{\frac{m}{2}} \frac{n}{n-j} \binom{n-j}{j}\binom{j}{k-j},\]
with the usual convention that $\binom{i}{j}=0$ if $i < j$ or $j < 0$. Therefore, it is sufficient to show that $\mbox{for } 1 \leq j \leq m/2$,
\[ \lim_{n\to\infty} \frac{\frac{n+1}{n+1-j}}{\frac{n}{n-j}} = 1 \mbox{ and } \lim_{n\to\infty} \frac{\binom{n+1-j}{j}}{\binom{n-j}{j}} = 1.\]
 
It is clear that $\lim_{n\to\infty} \frac{n+1}{n+1-j}/\frac{n}{n-j} = 1$. Furthermore,
\begin{align*}
    \frac{\binom{n+1-j}{j}}{\binom{n-j}{j}} & = \frac{(n+1-j)\cdots (n-2j+2)}{(n-j)\cdots(n-2j+1)}.
\end{align*}
It is clear that $\lim_{n\to\infty} \frac{n+1-j-\ell}{n-j-\ell} = 1$ for $0 \leq \ell \leq j-1 $. Therefore, 
\[\lim_{n\to\infty} \frac{\binom{n+1-j}{j}}{\binom{n-j}{j}} = 1,\]
concluding the proof. 
\endproof

\subsection{Proof of \cref{prop:sturfmels-thomas}}

\noindent \proof 
Consider the dual of $LP_=(A,c,b)$:
\begin{alignat*}{3}
 D_{=}(A,c,b):\quad\quad & \text{minimize} & b^\top y& \\
 &  \text{s.t.} \quad& y^\top A \geq c.
\end{alignat*}

\noindent We start by recalling the complementary slackness conditions. If $x$ and $y$ are feasible solutions to the primal and dual problem, respectively, then complementary slackness states that $x$ and $y$ are optimal solutions to their respective problems if and only if
\begin{alignat*}{4}\tag{CS}
  y_i(a_i^\top x-b_i) & = & \ 0, \ & \quad \forall \ i \in [m],\\
   (c_j - y^\top A_j) x_j & = & \ 0, \ & \quad \forall \ j \in [n].
\end{alignat*}
Let $x$ be an optimal solution of $LP_=(A,c,b)$ and $y$ be an optimal solution of $D_{=}(A,c,b)$. By complementary slackness, $x_j >0$ implies $y^\top A_j = c_j$, which means that the support of $x$ lies in a cell of $\Delta_c$. Conversely, let $x$ be a solution to \cref{eq:support-subset-of-cell}. Then there exists $y \in \R^m$ such that $\supp(x)\subset \{j \mid  y^\top A_j = c_j\}$. This implies that $c^\top x = y^\top Ax = y^\top b$, 
and hence $x$ is an optimal solution to $LP_=(A,c,b)$, due to strong duality. 
\endproof

\subsection{Proof of \cref{lem:pointedcone}}\label{appx:lemma_pointed_cone}

Let $\mathcal{K}$ be a  pointed $m$-dimensional cone, then there exists a vector $\gamma \in \R^n$ such that $\mathcal{K} \subset \{ x \in \R^n \mid \gamma^\top x \geq 0\}$ and $\mathcal{K} \cap \{ x \in \R^n \mid \gamma^\top x = 0\} = \textbf{0}$. Consider the hyperplane $\mathcal{H} = \{ x \in \R^n \mid \gamma^\top x = 1\}$. The set $\mathcal{H} \cap \mathcal{K}$ consists of more than just one point, and is a bounded section of $\mathcal{K}$. Therefore, $\mathcal{H} \cap \mathcal{K}$ is an $(m-1)$-dimensional polytope, whose vertices are determined by
the generators of $\mathcal{K}$. Now, consider a triangulation $\mathcal{T}$ of $\mathcal{H} \cap \mathcal{K}$. Every simplex $S_i \in \mathcal{T}$ gives rise
to a simplicial cone $\mathcal{K}_i = \cone(S_i)$. These simplicial cones, by construction, triangulate $\mathcal{K}$.  
\endproof

\subsection{Proof of \cref{lem:hyperplaneequation} and \cref{lemma:remaining_conditions}}\label{appendix:proof_lemmas}


The equation of a hyperplane can be derived from computing determinants of the form
\begin{equation*}\label{eq:determinant}
 \det\begin{pmatrix}
1 & \ldots & 1 & 1\\
v'_m(t_{i_1}) & \ldots & v'_m(t_{i_m}) & y \end{pmatrix}.
\end{equation*}
We present results that link the determinant above to the determinant that defines the facets of the cyclic polytope, and where the Vandermonde determinant shows up. We start by stating the known result that the Vandermonde matrix is totally positive. We then show that the sub-determinants of $A'$ have the same absolute value of the sub-determinants of the Vandermonde matrix.
\begin{claim}\label{clm:totalpositivity}
The Vandermonde matrix 
\begin{equation*}
B = \begin{pmatrix}
1 & \ldots & 1\\
v_m(t_1) & \ldots & v_m(t_n)
\end{pmatrix}
\end{equation*}
is totally positive, i.e., all square submatrices of size at most $m+1$ have strictly positive determinants.

\end{claim}

\begin{proof}
This work \cite{fekete1912problem} proves that a sufficient condition for total positivity is that all solid minors have positive determinants. A minor is called solid if the indices of its rows and columns are consecutive. If this is applied to a Vandermonde matrix, then positivity of solid minors follows from the formula of the Vandermonde determinant, up to factoring out the appropriate (positive) scaling of each row.
\end{proof}
 
\begin{claim}\label{clm:determinant} 
Let $0 < t_{j_1} < \ldots < t_{j_{m+1}}$ with $j_1<\cdots<j_{m+1}$, we have
\begin{equation} \label{eq:sign}
\det\begin{pmatrix}
1 & \ldots & 1\\
v'_m(t_{j_1}) & \ldots & v'_m(t_{j_{m+1}})
\end{pmatrix} = (-1)^{\lfloor \frac{m}{2}\rfloor} \det \begin{pmatrix}
1 & \ldots & 1\\
v_m(t_{j_1}) & \ldots & v_m(t_{j_{m+1}})
\end{pmatrix}.
\end{equation} 
\end{claim}

\begin{proof}
The matrix on the left of \cref{eq:sign} can be obtained from the matrix on the right through a series of linear operations. First we multiply exactly $\lfloor \frac{m}{2}\rfloor$ rows by -1 (if $m$ is odd, row 2, 4, 6, $\ldots$; if $m$ is even, row 1, 3, 5, $\ldots$), this multiplies the determinant by $(-1)^{\lfloor \frac{m}{2}\rfloor}$. Then we multiply the first row by $M$ and add it to these rows. This last operation does not change the determinant.
\end{proof}

\begin{claim}\label{clm:beta}
Let $0 < t_1 < \ldots < t_{m}$. If $m$ is odd,
\begin{equation*}
\mathrm{sign} \det\begin{pmatrix}
v'_m(t_1) & \ldots & v'_m(t_{m})
\end{pmatrix} = \mathrm{sign} (-1)^{\lfloor \frac{m}{2}\rfloor}.
\end{equation*}
If $m$ is even, 
\begin{equation*}
\mathrm{sign} \det\begin{pmatrix}
v'_m(t_1) & \ldots & v'_m(t_{m})
\end{pmatrix} = \mathrm{sign} (-1)^{\lfloor \frac{m}{2}\rfloor+1}.
\end{equation*}
\end{claim}

\begin{proof}
Let $t_0 < t_1$, and $D = \det\begin{pmatrix}
v'_m(t_1) & \ldots & v'_m(t_{m})
\end{pmatrix}$. If $m$ is odd, by developing the first column of the following determinant we establish that:
\begin{align*}
\det\begin{pmatrix}
1 & 1 & \ldots & 1\\
v'_m(t_0) & v'_m(t_1) &  \ldots & v'_m(t_{m})
\end{pmatrix} & =   D - t_0 \cdot \det \begin{pmatrix}
1 & \ldots & 1\\
M - t_1^2 & \ldots & M - t_m^2\\
t_1^3 & \ldots & t_m^3\\
\vdots & \ldots & \vdots 
\end{pmatrix} 
\\
& \quad  + (M-t_0^2)\cdot \det \begin{pmatrix} 1 & \ldots & 1\\
t_1 & \ldots & t_m\\
t_1^3 & \ldots & t_m^3\\
M - t_1^4 & \ldots & M - t_m^4\\
\vdots & \ldots & \vdots\end{pmatrix}  + \cdots.
\end{align*} 

By a similar argument to \cref{clm:determinant}, we see that 

\begin{equation*}
   \det  \begin{pmatrix}
1 & \ldots & 1\\
M - t_1^2 & \ldots & M - t_m^2\\
t_1^3 & \ldots & t_m^3\\
\vdots & \ldots & \vdots 
\end{pmatrix} = (-1)^{\lfloor \frac{m}{2} \rfloor} \det \begin{pmatrix}
1 & \ldots & 1\\
 t_1^2 & \ldots & t_m^2\\
t_1^3 & \ldots & t_m^3\\
\vdots & \ldots & \vdots 
\end{pmatrix};
\end{equation*}

$ $
\begin{equation*}
   \det \begin{pmatrix} 1 & \ldots & 1\\
t_1 & \ldots & t_m\\
t_1^3 & \ldots & t_m^3\\
M - t_1^4 & \ldots & M - t_m^4\\
\vdots & \ldots & \vdots\end{pmatrix} = -(-1)^{\lfloor \frac{m}{2} \rfloor} \det \begin{pmatrix} 1 & \ldots & 1\\
t_1 & \ldots & t_m\\
t_1^3 & \ldots & t_m^3\\
t_1^4 & \ldots & t_m^4\\
\vdots & \ldots & \vdots\end{pmatrix}.
\end{equation*}

Using the total positivity from \cref{clm:totalpositivity}, 
\begin{equation}\label{eq:sign-second}
    \begin{split}
        \det\begin{pmatrix}
1 & 1 & \ldots & 1\\
v'_m(t_0) & v'_m(t_1) &  \ldots & v'_m(t_{m})
\end{pmatrix}  =    D & -(-1)^{\lfloor \frac{m}{2} \rfloor} t_0 \lambda_1  -(-1)^{\lfloor \frac{m}{2} \rfloor} (M-t_0)^2 \lambda_2\\ &-(-1)^{\lfloor \frac{m}{2} \rfloor} t_0^3 \lambda_3 \ldots,
    \end{split}
\end{equation}
where $\lambda_i > 0$ for $i \in [m]$. By \cref{clm:determinant}, the sign of the determinant on the left of \cref{eq:sign-second} is equal to the sign of $(-1)^{\lfloor \frac{m}{2} \rfloor}$. Therefore, by isolating $D$ in \cref{eq:sign-second}, $D$ can be expressed as the sum of $m+1$ terms with signes equal to $(-1)^{\lfloor \frac{m}{2} \rfloor}$. Therefore, the sign of $D$ is equal to $(-1)^{\lfloor \frac{m}{2} \rfloor}$.

If $m$ is even, we have a different construction of $A$. Expanding the second column of the following determinant we establish that:
\begin{align*}
\det\begin{pmatrix}
1 & 1 & \ldots & 1\\
v'_m(t_0) & v'_m(t_1) &  \ldots & v'_m(t_{m})
\end{pmatrix} & =   -D + (M-t_1) \cdot \det \begin{pmatrix}
1 & \ldots & 1\\
t_0^2 & \ldots & t_m^2\\
M-t_0^3 & \ldots & M-t_m^3\\
\vdots & \ldots & \vdots 
\end{pmatrix} 
\\
& \quad  - t_1^2\cdot \det \begin{pmatrix} 1 & \ldots & 1\\
M-t_0 & \ldots & M-t_m\\
M-t_0^3 & \ldots & M-t_m^3\\
t_0^4 & \ldots & t_m^4\\
\vdots & \ldots & \vdots\end{pmatrix}  + \cdots.
\end{align*} 

By a similar argument to \cref{clm:determinant}, we see that 

\begin{equation*}
   \det  \begin{pmatrix}
1 & \ldots & 1\\
t_0^2 & \ldots & t_m^2\\
M-t_0^3 & \ldots & M-t_m^3\\
\vdots & \ldots & \vdots 
\end{pmatrix} = (-1)^{\frac{m}{2}-1 } \det \begin{pmatrix}
1 & \ldots & 1\\
 t_0^2 & \ldots & t_m^2\\
t_0^3 & \ldots & t_m^3\\
\vdots & \ldots & \vdots 
\end{pmatrix};
\end{equation*}

$ $
\begin{equation*}
   \det \begin{pmatrix} 1 & \ldots & 1\\
M-t_0 & \ldots & M-t_m\\
M-t_0^3 & \ldots & M-t_m^3\\
t_0^4 & \ldots & t_m^4\\
\vdots & \ldots & \vdots\end{pmatrix} = (-1)^{ \frac{m}{2} } \det \begin{pmatrix} 1 & \ldots & 1\\
t_0 & \ldots & t_m\\
t_0^3 & \ldots & t_m^3\\
t_0^4 & \ldots & t_m^4\\
\vdots & \ldots & \vdots\end{pmatrix}.
\end{equation*}

Using the total positivity from \cref{clm:totalpositivity}, 
\begin{equation}\label{eq:sign-second-even}
\begin{split}
    \det\begin{pmatrix}
1 & 1 & \ldots & 1\\
v'_m(t_0) & v'_m(t_1) &  \ldots & v'_m(t_{m})
\end{pmatrix}  =   -D & -(-1)^{\frac{m}{2}} (M-t_1) \lambda_1  -(-1)^{\frac{m}{2}} t_1^2 \lambda_2\\ &-(-1)^{\frac{m}{2}} (M-t_1^3) \lambda_3 \ldots,
\end{split}
\end{equation}
where $\lambda_i > 0$ for $i \in [m]$. By \cref{clm:determinant}, the sign of the determinant on the left of \cref{eq:sign-second-even} is equal to the sign of $(-1)^{\frac{m}{2} }$. Therefore, by isolating $D$ in \cref{eq:sign-second-even}, $D$ can be expressed as the sum of $m+1$ terms all of sign equal to $(-1)^{\frac{m}{2}+1}$. Therefore, the sign of $D$ is equal to $(-1)^{ \frac{m}{2}+1 }$.
\end{proof}

We are now ready to present the proof of \cref{lem:hyperplaneequation}.

\begin{proof}[Proof of \cref{lem:hyperplaneequation}]
When $m$ is odd, expanding along the last column of the determinant in \cref{eq:hyperplane}, subtracting $M \times$ first row as necessary to clear the $M$'s in the $v'_m(t)$ in each sub-determinant that appears in the expansion, we get for any $k \in [m],$
\[ \alpha_k  = (-1)^{k+m}  \det\begin{pmatrix}
1 & \ldots & 1\\
(-1)^{1+1}t_{i_1} & \ldots & (-1)^{1+1}t_{i_m}\\
\vdots & \ldots &\vdots\\
(-1)^{k}t_{i_1}^{k-1} & \ldots & (-1)^{k}t_{i_m}^{k-1}\\
(-1)^{k+2}t_{i_1}^{k+1} & \ldots & (-1)^{k+2}t_{i_m}^{k+1}\\
\vdots & \ldots &\vdots\\
(-1)^{m+1}t_{i_1}^{m} & \ldots & (-1)^{m+1}t_{i_m}^{m}\\
\end{pmatrix} =  (-1)^{k+m}  (-1)^{\lfloor\frac{m}{2}\rfloor + k+1} \det\begin{pmatrix}
1 & \ldots & 1\\
t_{i_1} & \ldots & t_{i_m}\\
\vdots & \ldots &\vdots\\
t_{i_1}^{k-1} & \ldots & t_{i_m}^{k-1}\\
t_{i_1}^{k+1} & \ldots & t_{i_m}^{k+1}\\
\vdots & \ldots &\vdots\\
t_{i_1}^{m} & \ldots & t_{i_m}^{m}\\
\end{pmatrix},\]
where the determinant in the far right is a minor of the Vandermonde matrix and is therefore positive by \cref{clm:totalpositivity}. Hence, for $k \in [m]:$
\[ \mbox{sign}(\alpha_k) = (-1)^{k+m}  (-1)^{\lfloor\frac{m}{2}\rfloor + k+1} = (-1)^{\lfloor\frac{m}{2}\rfloor + m+1}.\]

By Laplace expansion we also get,
\[ \beta = (-1) \cdot (-1)^{m} \det\begin{pmatrix}
v'_m(t_{i_1}) & \ldots & v'_m(t_{i_m}) \end{pmatrix}.\]

By \cref{clm:beta},  $\mbox{sign}(\det\begin{pmatrix}
v'_m(t_{i_1}) & \ldots & v'_m(t_{i_m}) \end{pmatrix})  = (-1)^{\lfloor \frac{m}{2} \rfloor}$ and, therefore,
\[ \mbox{sign}(\beta) = (-1)^{\lfloor \frac{m}{2} \rfloor + m + 1}= (-1)^{\lfloor \frac{m}{2} \rfloor}.  \]

When $m$ is even, by Laplace expanding on the last column of the determinant in \cref{eq:hyperplane}, and subtracting $M \times$ first row as necessary to clear the $M$'s in the $v'_m(t)$ in each sub-determinant that appears in the expansion, we get for any $k \in [m],$
\[ \alpha_k  =  (-1)^{\frac{m}{2}}  \det\begin{pmatrix}
1 & \ldots & 1\\
t_{i_1} & \ldots & t_{i_m}\\
\vdots & \ldots &\vdots\\
t_{i_1}^{k-1} & \ldots & t_{i_m}^{k-1}\\
t_{i_1}^{k+1} & \ldots & t_{i_m}^{k+1}\\
\vdots & \ldots &\vdots\\
t_{i_1}^{m} & \ldots & t_{i_m}^{m}\\
\end{pmatrix},\]
where the determinant in the far right is a minor of the Vandermonde matrix and is therefore positive by \cref{clm:totalpositivity}. Hence, for $k \in [m]:$
\[ \mbox{sign}(\alpha_k) = (-1)^{\frac{m}{2}}.\]

By Laplace expansion we also get,
\[ \beta = -\det\begin{pmatrix}
v'_m(t_{i_1}) & \ldots & v'_m(t_{i_m}) \end{pmatrix}.\]

By \cref{clm:beta},  $\mbox{sign}(\det\begin{pmatrix}
v'_m(t_{i_1}) & \ldots & v'_m(t_{i_m}) \end{pmatrix})  = (-1)^{\frac{m}{2}+1}$ and, therefore,
\[ \mbox{sign}(\beta) = (-1)^{\frac{m}{2}}.  \]
\end{proof}

\begin{proof}[Proof of \cref{lemma:remaining_conditions}]
Given the construction $A$ in \cref{sec:construction}, let $C$ be an even facet. Take an arbitrary $x\ge0$ with support equal to $C$, and define $b$ to equal $Ax$. We show that the support of the unique optimal solution to $L(A,c,b)$ is equal to $C$.

Let $y = \alpha/\beta$, we use $y$ as a certificate and show that $y$ and $x$ satisfy the complementary slackness conditions by showing that $y$ verifies \cref{def:inequalitycell}. By \cref{lem:hyperplaneequation},  $\beta$ and $\alpha$ have the same signs, and by the total positivity of the Vandermonde matrix, $\beta \neq 0$ and $\alpha_i \neq 0$ for $i \in [m]$. Therefore,
\[ y_i > 0, \ \ \forall i \in [m]. \]

For $i \in C$,  
\[ y^\top v'_m(t_i) = \frac{\alpha^\top v'_m(t_i)}{\beta} = \frac{\beta}{\beta} = 1 = c_i.\]

Now, let $i \not\in C$,
\begin{center}
    
\begin{tabular}{lll}
\\
    $y^\top v'_m(t_i)$ & =  & $\frac{1}{\beta}( \alpha^\top v'_m(t_i) - \beta + \beta)$  \\
    \\
    & = & $\frac{1}{\beta}\det\begin{pmatrix}
1 & \ldots & 1 & 1\\
v'_m(t_{i_1}) & \ldots & v'_m(t_{i_m}) & v'_m(t_i) \end{pmatrix} + 1$ 
     \\
     \\
     & = & $\frac{1}{\beta} (-1)^g (-1)^{\lfloor \frac{m}{2} \rfloor} D+ 1,$
     \\
\end{tabular}
\end{center}
where \[ D = \det\begin{pmatrix}
1 & \ldots & 1 & \ldots & 1\\
v_m(t_{i_1}) & \ldots & v_m(t_i) & \ldots & v_m(t_{i_m})
\end{pmatrix}> 0, \]
and $i$ is inserted according to the correct increasing order between $i_1$ and $i_m$. In the third equality we used the fact that the permutation that put $i$ in the correct order has parity $g$ by definition of $C$ and $g$. Therefore, to show that $y^\top v'_m(t_i)>1$ we only need to show that $\frac{1}{\beta} (-1)^{g+ \lfloor \frac{m}{2} \rfloor} D >~0$.

By \cref{lem:hyperplaneequation} and \cref{clm:beta}, 
\[ \beta = (-1)^{\lfloor \frac{m}{2} \rfloor} |\det\begin{pmatrix}
 v'_m(t_{i_1}) & \ldots & v'_m(t_{i_m}) \end{pmatrix}| = (-1)^{\lfloor \frac{m}{2} \rfloor} E,\]
where $E = |\det\begin{pmatrix}
 v'_m(t_{i_1}) & \ldots & v'_m(t_{i_m}) \end{pmatrix}| > 0.$ Hence
\[ \frac{1}{\beta} (-1)^{g+ \lfloor \frac{m}{2} \rfloor} D =  (-1)^{g+ \lfloor \frac{m}{2} \rfloor} (-1)^{\lfloor \frac{m}{2} \rfloor} \frac{D}{E} = (-1)^{g} \frac{D}{E} > 0,\]
where the last inequality stems from $C$ being even (i.e., $g(C)=2$), and from $D > 0$, $E > 0$. Therefore, $y$ satisfies the following equations
\begin{alignat*}{4}
  y_i & > & 0 & , \quad \forall \ i \in [m],\\
  y^\top A_j & = & c_j & ,  \quad \forall \ j \in C,\\
  y^\top A_j & > & c_j & ,  \quad \forall \ j \not\in C.
\end{alignat*}
This shows that $y$ satisfies \cref{def:inequalitycell}, and implies that $C$ is in a loadout by \cref{lem:inequalitycell}. 
\end{proof}

\begin{claim}\label{lem:degeneracy}
If (P) has multiple optimal solutions then every optimal basic solution
to (D) is degenerate.
\end{claim}

\proof 
We show that if (D) has a nondegenerate optimal solution, then $(P)$ will have a unique optimal solution. Assume $y_1$ is an nondegenerate dual optimal point, thus by definition of a dual basic feasible solution, it satisfies exactly $m$ linear independent active constraints:
\begin{alignat*}{4}
  y_i & = & \ 0, \ & \quad \forall \ i \in M_1,\\
   (c_j - y^\top A_j) & = & \ 0, \ & \quad \forall \ j \in M_2,\\
   |M_1| + |M_2| & = & m. & 
\end{alignat*}
Consider an optimal primal solution $x$. The solution $x$ must satisfy the complementary slackness conditions. Consider $j \in [n] \setminus M_2$. We have $(c_j - y^\top A_j) > 0$ so we must have $x_j = 0$. Note also that $y_i \neq 0$ for $i \not\in M_1$. Therefore, $(a_i^\top x-b_i) = 0$ for $i \not\in M_1$.
This forms $m-|M_1| + n - |M_2| = n$ linear independent constraints and, therefore, an $n \times n$ matrix that uniquely determines $x$. 
\endproof

\subsection{Exact Tight Constructions for $m=3$ and $m=2$}\label{appx:exact_construction}
For $m= 3$ and $n>m$, \cref{thm:upperbound} establishes that $ \L^3(A,c) \leq 2n-5$ and $\L^2(A,c) \leq 3n-6$ for every design $(A,c)$. We now provide a construction of a design that matches both upper bounds.


\lowerboundSmallM*

\proof 
Let $n>m=3$, consider the following (inequality) design
\begin{align*}
c^\top &=\begin{pmatrix}
1 &1 &\sqrt{\frac{2}{3}} &\sqrt{\frac{2}{4}} &\cdots &\sqrt{\frac{2}{n}}
\end{pmatrix} \in \R^n,
\end{align*}
\begin{align*}
A &=\begin{pmatrix}
1 &0 &\frac{1}{3} &\frac{1}{4} &\cdots &\frac{1}{n} \\
0 &1 &\frac{1}{3} &\frac{1}{4} &\cdots &\frac{1}{n} \\
1 &1 &1 &1 &\cdots &1
\end{pmatrix} \in \R^{3 \times n}.
\end{align*}
We index the columns of $c$ and $A$ by $1,\ldots,n$ from left to right. We claim  that all of the following $2n-5$ subsets of indices are inequality cells:
\begin{itemize}
\item $\{1,j,j+1\}$ for all $j=3,\ldots,n-1$ ($n-3$ loadouts)
\item $\{2,j,j+1\}$ for all $j=3,\ldots,n-1$ ($n-3$ loadouts)
\item $\{1,2,3\}$ (1 loadout)
\end{itemize}

By \cref{lem:inequalitycell}, this will imply that the design $(A,c)$ has $2n-5$ loadouts of size 3, and $n-1 + n-2 + n-3 = 3n -6$ loadouts of size 2. Note that the loadouts of size 2 are as follows:
\begin{itemize}
\item $\{1,j\}$ for all $j=2,\ldots,n$ ($n-1$ loadouts)
\item $\{2,j\}$ for all $j=3,\ldots,n$ ($n-2$ loadouts)
\item $\{j,j+1\}$ for all $j=3,\ldots,n-1$ ($n-3$ loadouts)

\end{itemize}

Consider $j \in \{3,\ldots,n-1\}$. To show that $\{1,j,j+1\}$ is a loadout, we show that $\{1,j,j+1\}$ is an inequality cell by solving the system 
\begin{alignat}{4}\label{eq:system}
y_i & > & 0 & , \quad \forall \ i \in \{1,2,3\};\\
\nonumber y^\top A_{\ell} & = & c_{\ell} & ,  \quad \forall \ \ell \in \{1,j,j+1\};\\
\nonumber y^\top A_{\ell} & > & c_{\ell} & ,  \quad \forall \ \ell \not\in \{1,j,j+1\}. 
\end{alignat}
The three equalities of \cref{eq:system} translate to \begin{align*}
y_1 +y_3 &=1,
\\ y_1 +y_2 +jy_3 &=\sqrt{2j},
\\ y_1 +y_2 +(j+1)y_3 &=\sqrt{2(j+1)}.
\end{align*}
By solving for $y$,
\begin{align*}
y_3 & = \sqrt{2(j+1)}-\sqrt{2j} > 0,\\
y_1 & = 1 - (\sqrt{2(j+1)}-\sqrt{2j}) > 0,\\
y_2 & = \sqrt{2(j+1)} - 1 - j(\sqrt{2(j+1)}-\sqrt{2j}) > 0.
\end{align*} 
Now, take any $\ell=3,\ldots,n$, we show that $y^\top A_{\ell}\ge c_{\ell}$ with equality if and only if $\ell=j$ or $\ell=j+1$. Consider $\ell \in \{3,\ldots, n\}$, then
\begin{align}
y^\top A_{\ell}  \geq c_{\ell} & \Longleftrightarrow  \sqrt{2j}-jy_3+\ell y_3-\sqrt{2\ell} \ge0 \nonumber
\\ & \Longleftrightarrow (\ell-j)(\sqrt{2(j+1)}-\sqrt{2j}) \ge\sqrt{2\ell}-\sqrt{2j} \nonumber
\\ & \Longleftrightarrow (\ell-j)(\sqrt{j+1}-\sqrt{j}) \ge\sqrt{\ell}-\sqrt{j}. \label{eq:rhs_3}
\end{align}
It is clear that \cref{eq:rhs_3} is an equality when $\ell = j, j+1$. Suppose $\ell>j+1$, then the rhs of \cref{eq:rhs_3} can be written as \[(\sqrt{\ell}-\sqrt{\ell-1})+(\sqrt{\ell-1}-\sqrt{\ell-2})+\cdots+(\sqrt{j+1}-\sqrt{j}).\]
There are $\ell-j$ terms in parentheses. All these terms are less than $\sqrt{j+1}-\sqrt{j}$, and at least one of them is strictly less than $\sqrt{j+1}-\sqrt{j}$. Therefore the inequality \cref{eq:rhs_3} is strict and $y^\top A_{\ell}  > c_{\ell}$ when $\ell > j+1$. When $\ell < j$, \cref{eq:rhs_3} is equivalent to
\[ (j-\ell)(\sqrt{j+1}-\sqrt{j}) \le \sqrt{j}- \sqrt{\ell}.\]
The right-hand side of the last inequality can be written as \[(\sqrt{j}-\sqrt{j-1})+(\sqrt{j-1}-\sqrt{j-2})+\cdots+(\sqrt{\ell+1}-\sqrt{\ell}).\]
There are $\ell-j$ terms in parentheses. All these terms are greater than $\sqrt{j+1}-\sqrt{j}$, and at least one of them is strictly greater than $\sqrt{j+1}-\sqrt{j}$. Therefore the inequality \cref{eq:rhs_3} is strict and $y^\top A_{\ell}  > c_{\ell}$ when $\ell < j$. Finally, we must check the case $\ell=2$. We have 
\begin{align*}
y^\top A_{2}  > c_{2} & \Longleftrightarrow  y_2 + y_3 > 0,
\end{align*}
and the right-hand side inequality is true since $y_3 > 0$. This shows that $\{1,j,j+1\}$ is an inequality cell. With the same argument, we can show that $\{2,j,j+1\}$ is an inequality cell for $j \in \{3,\ldots,n-1\}$.

To see that $\{1,2,3\}$ is an inequality cell, we solve the system $ y^\top A_{\ell} =  c_{\ell},\ \  \forall \ \ell \in \{1,2,3\}$, which is equivalent to
\begin{align*}
y_1 +y_3 &=1,
\\ y_2 +y_3& =1,
\\ y_1 +y_2 +3y_3 &=\sqrt{6}.
\end{align*}
Solving this system yields
\begin{align*}
y_1 & = 3 - \sqrt{6} > 0,\\
y_2 & =3 - \sqrt{6} > 0,\\
y_3 & = \sqrt{6}-2> 0.
\end{align*} 
Now, consider $\ell \in \{4,\ldots,n\}$, in which case
\begin{align}
y^\top A_{\ell}  > c_{\ell} & \Longleftrightarrow y_1 + y_2 + \ell y_3 > \sqrt{2\ell} \nonumber\\
& \Longleftrightarrow  \sqrt{6}+(\ell -3)(\sqrt{6}-2)-\sqrt{2\ell} > 0. \label{eq:sepcial_cell}
\end{align}
To see that the last inequality is true, we study the function $x \mapsto f(x) = \sqrt{6}+(x -3)(\sqrt{6}-2)-\sqrt{2x}$  for $x \geq 4$. The derivative of $f$ is 
\[ f'(x) = \sqrt{6}-2 - \frac{1}{\sqrt{2x}}.\]
It is easy to see that $f'(x) > 0 $ for $x \geq 4$. Therefore $f$ is increasing over $[4,\infty]$. Furthermore, $f(4) > 0$. This implies that $f(x) > 0$ for $x \geq 4,$ and that \cref{eq:sepcial_cell} is true for $\ell \in \{4,\ldots,n\}$. We, therefore, conclude that $\{1,2,3\}$ is an inequality cell for the design $(A,c)$. 
\endproof

For $m= 2$ and $n>m$, \cref{thm:upperbound} establishes that $ |\L^2(A,c)| \leq n-1$ for every design $(A,c)$. We  provide a construction of a design that matches this upper bound.

\lowerboundSmallMtwo*

\proof 
Let $n>m=2$, consider the following (inequality) design
\begin{align*}
c^\top &=\begin{pmatrix}
1 &2 &\cdots &n
\end{pmatrix} \in \R^n
\end{align*}
\begin{align*}
A &=\begin{pmatrix}
1^2 &2^2 &\cdots &n^2 \\
1 &1 &\cdots &1
\end{pmatrix} \in \R^{2 \times n}.
\end{align*}
We claim  that all of the $n-1$ subsets of indices of the form $\{j,j+1\}$ with $j \in \{1,\ldots,n-1\}$ are inequality cells. Consider $j \in \{1,\ldots,n-1\}$. To show that $\{j,j+1\}$ is an inequality cell, we solve the system 
\begin{alignat}{4}\label{eq:system_two}
y_i & > & 0 & , \quad \forall \ i \in \{1,2\};\\
\nonumber y^\top A_{\ell} & = & c_{\ell} & ,  \quad \forall \ \ell \in \{j,j+1\};\\
\nonumber y^\top A_{\ell} & > & c_{\ell} & ,  \quad \forall \ \ell \not\in \{j,j+1\}. 
\end{alignat}
The two equalities of \cref{eq:system_two} translate to \begin{align*}
 y_1\cdot j^2 +y_2  &=j,
\\ y_1\cdot (j+1)^2 +y_2 &= j+1.
\end{align*}
By solving for $y$,
\begin{align*}
y_1 & = \frac{1}{2j+1} > 0,\\
y_2 & =\frac{j^2 + j}{2j+1} > 0.
\end{align*} 
Now, take any $\ell\in [n] \setminus \{j,j+1\}$, we show that $y^\top A_{\ell}> c_{\ell}$.
\begin{align}
y^\top A_{\ell} > c_{\ell} & \Longleftrightarrow  y_1 \ell^2 + y_2 > \ell \nonumber
\\ & \Longleftrightarrow \frac{\ell^2 + j^2 +j}{2j+1} > \ell. \label{eq:rhs_32}
\end{align}
To see that the last inequality is true, we study the function $x \mapsto f(x) = \frac{x^2 + j^2 +j}{2j+1}-x$ over $[1,n]$. The derivative of $f$ is 
\[ f'(x) = \frac{2x}{2j+1}-1.\]
It is easy to see that $f'(x) < 0 $ for $x \leq j$ and $f'(x) > 0$ for $x \geq j+1$. Therefore $f$ is decreasing over $[1,j]$ and increasing over $[j+1,n]$. Furthermore, $f(j) = f(j+1) = 0$. This implies that $f(l) > 0 $ for $l \in \{1, \ldots, j-1,j+2,\ldots,n\}$, which proves \cref{eq:rhs_32}.  We, therefore, conclude that $\{j,j+1\}$ is an inequality cell for the design $(A,c)$. 
\endproof
\end{appendices}
\end{document}